\documentclass[11pt]{article}

\usepackage[latin1]{inputenc}
\usepackage[english]{babel}
\usepackage{here,enumerate,graphicx}
\usepackage{amsmath,amssymb,amscd,amsthm,mathrsfs}
\usepackage[colorlinks=true,linkcolor=blue,urlcolor=blue]{hyperref}

\usepackage[flushmargin]{footmisc}
\newcommand\blfootnote[1]{%
	\begingroup
	\renewcommand\thefootnote{}\footnote{#1}%
	\addtocounter{footnote}{-1}%
	\endgroup}

\theoremstyle{plain}
\newtheorem{theorem}{Theorem}[section]
\newtheorem{proposition}[theorem]{Proposition}

\newtheorem{lemma}[theorem]{Lemma}
\newtheorem{corollary}[theorem]{Corollary}

\theoremstyle{definition}

\newtheorem{remark}[theorem]{Remark}

\newtheorem{question}[theorem]{Question}

\numberwithin{equation}{section}

\newcommand{\NN}{\mathbb{N}}
\newcommand{\QQ}{\mathbb{Q}}
\newcommand{\RR}{\mathbb{R}}
\newcommand{\CC}{\mathbb{C}}
\newcommand{\TT}{\mathbb{T}}
\newcommand{\FF}{\mathbb{F}}
\newcommand{\II}{\mathbb{I}}

\newcommand{\Ac}{\mathcal{A}}
\newcommand{\AP}{\mathcal{AP}}
\newcommand{\Bc}{\mathcal{B}}
\newcommand{\Cc}{\mathcal{C}}
\newcommand{\Fc}{\mathcal{F}}
\newcommand{\Kc}{\mathcal{K}}
\newcommand{\Nc}{\mathcal{N}}
\newcommand{\Pc}{\mathcal{P}}
\newcommand{\Part}{\mathscr{P}}
\newcommand{\Tc}{\mathcal{T}}

\newcommand{\Per}{\operatorname{Per}}
\newcommand{\Rec}{\text{Rec}}

\newcommand{\Tran}{\text{Tran}}

\newcommand{\Bdsup}{\overline{\operatorname{Bd}}}
\newcommand{\dinf}{\underline{\operatorname{dens}}}
\newcommand{\BDsup}{\overline{\mathcal{BD}}}
\newcommand{\Dinf}{\underline{\mathcal{D}}}

\newcommand{\cl}{\overline}

\newcommand{\orb}{\mathcal{O}}
\newcommand{\res}{\arrowvert}
\newcommand{\eps}{\varepsilon}

\newcommand{\com}{\overline}
\newcommand{\fuz}{\hat}
\newcommand{\eend}{\operatorname{end}}
\newcommand{\send}{\operatorname{send}}

\begin{document}
\begin{center}
	\begin{LARGE}
		{\bf Topological dynamics for the endograph metric I: Equivalences with other metrics}
	\end{LARGE}
\end{center}

\begin{center}
	\begin{Large}
		Antoni L\'opez-Mart\'inez\blfootnote{\textbf{Journal-ref}: Iranian Journal of Fuzzy Systems, Volume 23, Number 2, (2026), 81--100.\\ \textbf{DOI}: \href{https://doi.org/10.22111/ijfs.2026.54067.9574}{https://doi.org/10.22111/ijfs.2026.54067.9574}}
	\end{Large}
\end{center}


\begin{abstract}
	Given a dynamical system $(X,f)$ we investigate several topological dynamical properties for its Zadeh extension $(\mathcal{F}(X),\hat{f})$ endowed with the endograph metric $d_{E}$. In particular, we prove that for topological $\mathcal{A}$-transitivity, topological $(\ell,\mathcal{A})$-recurrence, Devaney chaos, and the specification property, the endograph metric behaves similarly to the supremum metric~$d_{\infty}$, the Skorokhod metric~$d_{0}$ and the sendograph metric~$d_{S}$. Our results not only resolve certain open questions in the existing literature, but also yield completely new outcomes in terms of point-$\mathcal{A}$-transitivity.
\end{abstract}


\section{Introduction}

In this paper a {\em dynamical system} will be a pair $(X,f)$ formed by a continuous map $f:X\longrightarrow X$ acting on a metric space~$(X,d)$. We will be interested in the interplay between the dynamical properties presented by $(X,f)$ and those presented by the induced system $(\Fc(X),\fuz{f})$, where $\fuz{f}$ is the {\em Zadeh extension} (also called {\em fuzzification}) of $f$, acting on the space $\Fc(X)$ of normal fuzzy subsets of $X$.

Delving into the relations between the topological dynamical properties of $(X,f)$ and $(\Fc(X),\fuz{f})$ has been a matter of study in (at least) the following works:\\[-20pt]
\begin{enumerate}[--]
	\item Jard\'on et al.~studied in~\cite{JardonSan2021_FSS_expansive,JardonSan2021_IJFS_sensitivity,JardonSanSan2020_FSS_some,JardonSanSan2020_MAT_transitivity} some {\em expansive}, {\em sensitive}, {\em contractive} and {\em transitivity} notions;
	
	\item Mart\'inez-Gim\'enez et al.~explored in~\cite{MartinezPeRo2021_MAT_chaos} the concepts of {\em Devaney}, {\em Li-Yorke} and {\em distributional chaos};
	
	\item Bartoll et al.,~in~\cite{BartollMaPeRo2022_AXI_orbit}, focused on the so-called {\em specification} and {\em shadowing} properties;
	
	\item and \'Alvarez et al.,~in \cite{AlvarezLoPe2025_FSS_recurrence}, looked for various forms of {\em recurrence}.
\end{enumerate}
These works have studied the system $(\Fc(X),\fuz{f})$ for four different metrics on $\Fc(X)$ at the same time: in \cite{AlvarezLoPe2025_FSS_recurrence,BartollMaPeRo2022_AXI_orbit,JardonSanSan2020_FSS_some,JardonSanSan2020_MAT_transitivity,MartinezPeRo2021_MAT_chaos} the {\em supremum metric}~$d_{\infty}$ and the {\em Skorokhod metric}~$d_{0}$ were considered, but also in~\cite{JardonSan2021_FSS_expansive,JardonSan2021_IJFS_sensitivity,JardonSanSan2020_MAT_transitivity} the {\em sendograph metric}~$d_{S}$ and the {\em endograph metric}~$d_E$ were used. In this paper, and following the already mentioned works, for each metric $\rho \in \{ d_{\infty} , d_{0} , d_{S} , d_{E} \}$ we will denote the metric space $(\Fc(X),\rho)$ by $\Fc_{\infty}(X)$, $\Fc_{0}(X)$, $\Fc_{S}(X)$, and $\Fc_{E}(X)$ respectively. Moreover, and also following the previous works, we will denote by $\tau_{\infty}$, $\tau_{0}$, $\tau_{S}$, and $\tau_{E}$ the respective topologies induced on $\Fc(X)$.

The usual hyperextension $(\Kc(X),\com{f})$ on the space $\Kc(X)$ of non-empty compact subsets of $X$ is also studied in \cite{AlvarezLoPe2025_FSS_recurrence,BartollMaPeRo2022_AXI_orbit,JardonSanSan2020_MAT_transitivity,MartinezPeRo2021_MAT_chaos}. The general kind of result obtained there for each dynamical property is the equivalence of having such a property between $(\Kc(X),\com{f})$, $(\Fc_{\infty}(X),\fuz{f})$ and $(\Fc_{0}(X),\fuz{f})$ in~\cite{AlvarezLoPe2025_FSS_recurrence,BartollMaPeRo2022_AXI_orbit,MartinezPeRo2021_MAT_chaos}, but also the equivalence with $(\Fc_{S}(X),\fuz{f})$ in \cite{JardonSanSan2020_MAT_transitivity}. However, the systems $(\Fc_{S}(X),\fuz{f})$ and $(\Fc_{E}(X),\fuz{f})$ were not considered in \cite{AlvarezLoPe2025_FSS_recurrence,BartollMaPeRo2022_AXI_orbit,MartinezPeRo2021_MAT_chaos}, and the following problem was indirectly stated in \cite[Page~7]{JardonSanSan2020_MAT_transitivity}:

\begin{question}\label{Q:main1}
	Let $(X,f)$ be a dynamical system and assume that the induced fuzzy system $(\Fc_{E}(X),\fuz{f})$ is {\em topologically transitive}. Does it follow that $(\Kc(X),\com{f})$ is {\em topologically transitive}?
\end{question}

Since \cite[Theorem 3]{JardonSanSan2020_MAT_transitivity} states that the notion of {\em topological transitivity} is presented by any and all of the extended systems $(\Kc(X),\com{f})$, $(\Fc_{\infty}(X),\fuz{f})$, $(\Fc_{0}(X)\fuz{f})$ and $(\Fc_{S}(X),\fuz{f})$ if and only if the original dynamical system $(X,f)$ is {\em weakly-mixing}, and since \cite[Proposition 7]{JardonSanSan2020_MAT_transitivity} shows that when $(\Fc_{S}(X),\fuz{f})$ is {\em topologically transitive} then so is the system $(\Fc_{E}(X),\fuz{f})$, Question~\ref{Q:main1} can be rewritten as follows:

\begin{question}\label{Q:main2}
	Is a dynamical system $(X,f)$ {\em weakly-mixing}, or are $(\Kc(X),\com{f})$ $(\Fc_{\infty}(X),\fuz{f})$, $(\Fc_{0}(X),\fuz{f})$ and $(\Fc_{S}(X),\fuz{f})$ {\em topologically transitive}, if and only if $(\Fc_{E}(X),\fuz{f})$ is {\em topologically transitive}?
\end{question}

In view of Question~\ref{Q:main2} let us recall that, given a (not necessarily {\em topological transitivity}) property~$\Pc$ among those studied in~\cite{AlvarezLoPe2025_FSS_recurrence,BartollMaPeRo2022_AXI_orbit,JardonSanSan2020_MAT_transitivity,MartinezPeRo2021_MAT_chaos}, then the first step to establish the desired equivalences is proving that ``\textit{if the system $(\Kc(X),\com{f})$ has property $\Pc$, then $(\Fc_{\infty}(X),\fuz{f})$ has property $\Pc$}''. Thus, since each property~$\Pc$ depends on ``\textit{certain conditions that \textbf{all open subsets} fulfill}'', and since for any metric space $(X,d)$ we have that $\tau_{E} \subset \tau_{S} \subset \tau_{0} \subset \tau_{\infty}$ on $\Fc(X)$, the only difficulty in showing that~$\Pc$ is equivalent for $(\Kc(X),\com{f})$, $(\Fc_{\infty}(X),\fuz{f})$, $(\Fc_{0}(X),\fuz{f})$, $(\Fc_{S}(X),\fuz{f})$ and $(\Fc_{E}(X),\fuz{f})$ is proving that
\begin{equation}\label{eq:d_E->d_H}
	\text{``\textit{if the system $(\Fc_{E}(X),\fuz{f})$ has property $\Pc$, then $(\Kc(X),\com{f})$ has property $\Pc$}''.}\tag{$*$}
\end{equation}
The main objective of this paper is to provide a rather general argument to prove statement \eqref{eq:d_E->d_H} for most of the topological dynamical properties considered in the works~\cite{AlvarezLoPe2025_FSS_recurrence,BartollMaPeRo2022_AXI_orbit,JardonSanSan2020_MAT_transitivity,MartinezPeRo2021_MAT_chaos}. As a consequence we are able to answer Question~\ref{Q:main1} (and hence Question~\ref{Q:main2}) in the positive, and we also extend most of the results from~\cite{AlvarezLoPe2025_FSS_recurrence,BartollMaPeRo2022_AXI_orbit,JardonSanSan2020_MAT_transitivity,MartinezPeRo2021_MAT_chaos} by adding the remaining sendograph and endograph equivalences.

In particular: with Theorem~\ref{The:A-tran} we improve \cite[Theorem~3~and~Proposition~7]{JardonSanSan2020_MAT_transitivity} and \cite[Theorem~2]{MartinezPeRo2021_MAT_chaos}; with Theorem~\ref{The:(l,A)-rec} we extend \cite[Theorem~4.1~and~Corollary~4.1]{AlvarezLoPe2025_FSS_recurrence}; with Corollary~\ref{Cor:point-rec/tran} we improve the results \cite[Corollary~5.1]{AlvarezLoPe2025_FSS_recurrence} and \cite[Theorem~4]{JardonSanSan2020_MAT_transitivity}; with Lemma~\ref{Lem:point-A-rec/tran} we extend \cite[Proposition~9]{JardonSanSan2020_MAT_transitivity}; with both Lemma~\ref{Lem:periodic} and Theorem~\ref{The:Devaney} we improve \cite[Proposition~2,~Corollary~1~and~Theorem~1]{MartinezPeRo2021_MAT_chaos}; and Theorem~\ref{The:sp} and Corollary~\ref{Cor:sp} extend \cite[Theorem~2, Corollary~1 and Theorem~3]{BartollMaPeRo2022_AXI_orbit}. We also obtain some completely new outcomes in terms of {\em point-$\Ac$-transitivity} (see Theorem~\ref{The:point-A-tran} and Corollary~\ref{Cor:point-BD-tran}).

The paper is organized as follows. In Section~\ref{Sec_2:notation} we introduce the general background on fuzzy sets together with the key lemma that will allow us to obtain \eqref{eq:d_E->d_H} for most of the aforementioned dynamical properties (see Lemma~\ref{Lem:key} below). Then we divide the generalizations obtained into two sections: in~Section~\ref{Sec_3:t.r.Ff} we look at several types of {\em transitivity} and {\em recurrence} notions, and in~Section~\ref{Sec_4:chaos_sp} we focus on {\em Devaney chaos} and the so-called {\em specification property}. Finally, we summarize the improvements obtained in Section~\ref{Sec_5:conclusions}, where we also establish some possible future lines of research.

\section{General background and a key lemma}\label{Sec_2:notation}

In this section we introduce the fuzzy results that we need in the sequel: first, we recall the definition of~$\Fc(X)$ and~$\fuz{f}$ for a given dynamical system $(X,f)$; second, we recall the definitions of the different metrics that we use along the paper; and third, we prove the key lemma that allows us to develop the results stated at the Introduction (see Lemma~\ref{Lem:key} below). From now on let $\NN$ be the set of strictly positive integers, set $\NN_0 := \NN \cup \{0\}$, and let us denote the unit interval $[0,1]$ with the symbol $\II$.

\subsection[The space of normal fuzzy sets F(X)]{The space of normal fuzzy sets $\Fc(X)$}

A {\em fuzzy~set} on a topological space $X$ is a function $u:X\longrightarrow\II$, where the value $u(x) \in \II$ denotes the degree of membership of the point $x$ in $u$. For each fuzzy set $u$, its {\em $\alpha$-level} is denoted by
\[
u_{\alpha} := \{ x \in X \ ; \ u(x) \geq \alpha \} \ \text{ for each } \alpha \in \ ]0,1] \quad \text{ and } \quad u_0 := \cl{\{ x \in X \ ; \ u(x)>0 \}}.
\]

We denote by $\Fc(X)$ the {\em hyperspace of normal fuzzy sets} on $X$, i.e.\ the space formed by the fuzzy sets~$u$ that are upper-semicontinuous functions and such that~$u_0$ is compact and~$u_1$ is non-empty. For each compact subset $K \subset X$ we will denote by $\chi_K:X\longrightarrow\II$ the {\em characteristic function on $K$}, which belongs to $\Fc(X)$. Note also that the function $u := \max_{1\leq j\leq N} (\alpha_j \cdot \chi_{K_j})$ belongs to $\Fc(X)$ for each pair of finite sequences of values $0 < \alpha_1 < \alpha_2 < ... < \alpha_N = 1$ and compact subsets $K_1,...,K_N \subset X$, and that for such a fuzzy set $u$ we have the equality
\[
u_{\alpha} = \bigcup \{ K_j \ ; \ 1\leq j\leq N \text{ with } \alpha_j \geq \alpha \} \quad \text{ for each } \alpha \in \II.
\]
A kind of converse is also true, when $X$ is a metric space, as we recall in Lemma~\ref{Lem:eps.pisos} below.

It follows from \cite[Propositions~3.1 and 4.9]{JardonSanSan2020_FSS_some} that every continuous map $f:X\longrightarrow X$ acting on a Hausdorff topological space $X$ induces a well-defined map $\fuz{f}:\Fc(X)\longrightarrow\Fc(X)$, called the {\em Zadeh extension} (sometimes referred as the {\em fuzzification}) of $f$. We have that $\hat{f}$ maps each normal fuzzy set $u \in \Fc(X)$ to the following normal fuzzy set:
\begin{equation*}
\fuz{f}(u):X\longrightarrow\II \quad \text{ where } \quad \left[\fuz{f}(u)\right](x) :=
\left\{
\begin{array}{lcc}
	\sup\{ u(y) \ ; \ y \in f^{-1}(\{x\}) \}, & \text{ if } f^{-1}(\{x\}) \neq \varnothing, \\[5pt]
	0, & \text{ if } f^{-1}(\{x\}) = \varnothing.
\end{array}
\right.
\end{equation*}
Moreover, when $X$ is a metric space then we have the following well-known properties of the Zadeh extension $\fuz{f}$, which we will use without citing them (see \cite{JardonSanSan2020_FSS_some,JardonSanSan2020_MAT_transitivity,RomanChal2008_CSF_some} for details):

\begin{proposition}
	Let $f:X\longrightarrow X$ be a continuous map acting on a metric space $(X,d)$, $u \in \Fc(X)$, $\alpha \in \II$, $n \in \NN_0$, and let $K$ be a compact subset of $X$. Then:
	\begin{enumerate}[{\em(a)}]
		\item $\left[ \fuz{f}(u) \right]_{\alpha} = f(u_{\alpha})$, i.e.\ the $\alpha$-level of the $\fuz{f}$-image coincides with the $f$-image of the $\alpha$-level;
		
		\item $\left( \fuz{f} \right)^n = \widehat{f^n}$, i.e.\ the composition of $\fuz{f}$ with itself $n$ times coincides with the fuzzification of $f^n$;
		
		\item $\fuz{f}\left( \chi_K \right) = \chi_{f(K)}$, i.e.\ the fuzzification $\fuz{f}$ is an extension of the original map $f$.
	\end{enumerate}
\end{proposition}

To talk about the continuity of $\fuz{f}:\Fc(X)\longrightarrow\Fc(X)$ we will endow $\Fc(X)$ with the several fuzzy metrics mentioned in the Introduction (see Subsection~\ref{SubSec_2.2:metrics}). However, these fuzzy metrics are defined in terms of the so-called {\em Hausdorff distance}, so let us quickly recall here its precise definition. Given a metric space $(X,d)$ we will denote by $\Bc_d(x,\eps) \subset X$ the open $d$-ball centred at $x \in X$ and of radius $\eps>0$. Moreover, we will consider the spaces of sets
\[
\Cc(X) :=\left\{ C \subset X \ ; \ C \text{ is a non-empty closed set} \right\},
\]
\[
\Kc(X) := \left\{ K \subset X \ ; \ K \text{ is a non-empty compact set} \right\},
\]
and given two sets $C_1,C_2 \in \Cc(X)$ the {\em Hausdorff distance} between them is denoted and defined as
\[
d_H(C_1,C_2) := \max\left\{ \sup_{x_1 \in C_1} d(x_1,C_2) , \sup_{x_2 \in C_2} d(x_2,C_1) \right\}.
\]
This value may be infinite if either $C_1$ or $C_2$ are not compact. However, in $\Kc(X)$ we have that the map~$d_H:\Kc(X)\times\Kc(X)\longrightarrow[0,+\infty[$ defined as before is a metric, called the {\em Hausdorff metric}, and every continuous map $f:X\longrightarrow X$ induces a $d_H$-continuous map $\com{f}:\Kc(X)\longrightarrow\Kc(X)$, defined as~$\com{f}(K) := f(K) = \{ f(x) \ ; \ x \in K \}$ for each $K \in \Kc(X)$. This will be an extremely useful in-between system when comparing $(X,f)$ and $(\Fc(X),\fuz{f})$. It is also well-known that the topology induced in the space $\Kc(X)$ by $d_H$ coincides with the so-called {\em Vietoris topology}, but we will not delve into this further here. We will denote by $\Bc_H(K,\eps) \subset \Kc(X)$ the open $d_H$-ball centred at $K \in \Kc(X)$ and of radius $\eps>0$.\newpage

Given $C \in \Cc(X)$ and $\eps\geq 0$ we will write $C+\eps := \{ x \in X \ ; \ d(x,c) \leq \eps \text{ for some } c \in C \}$. The following are well-known (but very useful) facts:

\begin{proposition}\label{Pro:Hausdorff}
	Let $(X,d)$ be a metric space, $\eps\geq 0$, and let $C_1,C_2,C_3,C_4 \in \Cc(X)$. Then:
	\begin{enumerate}[{\em(a)}]
		\item We have that $d_H(C_1,C_2) \leq \eps$ if and only if $C_1 \subset C_2+\eps$ and $C_2 \subset C_1+\eps$.
		
		\item We always have that $d_H(C_1\cup C_2, C_3 \cup C_4) \leq \max\{ d_H(C_1,C_3) , d_H(C_2,C_4) \}$.
	\end{enumerate}
\end{proposition}

We refer the reader to \cite{IllanesNad1999_book_hyperspaces} for a detailed study of $\Cc(X)$ and $\Kc(X)$, and we conclude this subsection with another useful lemma regarding fuzzy normal sets and the Hausdorff metric (see \cite{JardonSanSan2020_FSS_some,JardonSanSan2020_MAT_transitivity,MartinezPeRo2021_MAT_chaos}):

\begin{lemma}\label{Lem:eps.pisos}
	Let $(X,d)$ be a metric space. Given any set $u \in \Fc(X)$ and $\eps>0$ there exist numbers $0 = \alpha_0 < \alpha_1 < \alpha_2 < ... < \alpha_N = 1$ such that $d_H(u_{\alpha}, u_{\alpha_{j+1}})<\eps$ for all $\alpha \in ]\alpha_j, \alpha_{j+1}]$ and $0\leq j\leq N-1$. In particular, since $d_H(u_{\alpha}, u_{\alpha_1})<\eps$ for every $\alpha \in \ ]0, \alpha_1]$, we also have that $d_H(u_0,u_{\alpha_1})\leq\eps$.
\end{lemma}

\subsection[The supremum, Skorokhod, sendograph and endograph metrics on F(X)]{The supremum, Skorokhod, sendograph and endograph metrics on $\Fc(X)$}\label{SubSec_2.2:metrics}

In this subsection we introduce the different fuzzy metrics that we consider along the paper. From now on we let $(X,d)$ be a metric space, and we will define the {\em supremum}, {\em Skorokhod}, {\em sendograph} and {\em endograph} metrics by using the original metric $d$ of $X$, but also the notion of Hausdorff distance.

Following \cite{JardonSanSan2020_MAT_transitivity} we start by the {\em supremum metric} $d_{\infty}:\Fc(X)\times\Fc(X)\longrightarrow[0,+\infty[$, which is defined for each pair $u,v \in \Fc(X)$ as
\[
d_{\infty}(u,v) := \sup_{\alpha \in \II} d_H(u_{\alpha},v_{\alpha}),
\]
where $d_H$ denotes the Hausdorff metric on $\Kc(X)$. We will denote by $\Bc_{\infty}(u,\eps)$ the open $d_{\infty}$-ball centred at $u \in \Fc(X)$ and of radius $\eps>0$, and as mentioned in the Introduction we will denote by $\Fc_{\infty}(X)$ the metric space $(\Fc(X),d_{\infty})$ and by $\tau_{\infty}$ the topology induced on $\Fc(X)$ by $d_{\infty}$. This metric has also been called the {\em level-wise metric} in \cite{Kupka2011_IS_on} (see also the references included there), and it is well-known that the metric space $\Fc_{\infty}(X)$ is non-separable whenever $X$ has more than one point (see \cite[Proposition~4.5]{JardonSan2021_IJFS_sensitivity}).

In the second place we have the {\em Skorokhod metric} $d_{0}:\Fc(X)\times\Fc(X)\longrightarrow[0,+\infty[$, which is defined for each pair $u,v \in \Fc(X)$ as
\[
d_{0}(u,v) := \inf\left\{ \eps>0 \ ; \ \text{there is } \xi \in \Tc \text{ with } \sup_{\alpha\in\II} |\xi(\alpha)-\alpha| \leq \eps \text{ and } d_{\infty}(u,\xi\circ v) \leq \eps \right\},
\]
where
\[
\Tc := \left\{ \xi:\II\longrightarrow\II \ ; \ \xi \text{ is a strictly increasing homeomorphism of } \II \right\}.
\]
We will denote by $\Bc_{0}(u,\eps)$ the open $d_{0}$-ball centred at $u \in \Fc(X)$ and of radius $\eps>0$, and we will denote by $\Fc_{0}(X)$ the metric space $(\Fc(X),d_{0})$ and by $\tau_{0}$ the topology induced on $\Fc(X)$ by $d_{0}$. This metric was introduced, for the case in which $(X,d)$ is an arbitrary metric space, in the 2020 paper \cite{JardonSanSan2020_FSS_some}. The relations between $d_{0}$ and other metrics in this rather general context have recently been studied in \cite{Huang2022_FSS_some}. Since the identity self-map of $\II$ belongs to $\Tc$, it is not hard to check that $d_{0}(u,v) \leq d_{\infty}(u,v)$ and also that $d_{0}(u,\chi_K) = d_{\infty}(u,\chi_K)$ for every $u,v \in \Fc(X)$ and $K \in \Kc(X)$.

To introduce the {\em endograph} and {\em sendograph metrics} we need to consider an auxiliary metric $\com{d}$ on the product space $X\times\II$ as follows:
\[
\com{d}\left((x,\alpha),(y,\beta)\right) := \max\left\{ d(x,y) , |\alpha - \beta | \right\} \quad \text{ for each pair } (x,\alpha),(y,\beta) \in X\times\II.
\]
Given now any fuzzy set $u \in \Fc(X)$, the {\em endograph of $u$} is defined as the following set
\[
\eend(u) := \left\{ (x,\alpha) \in X\times\II \ ; \ u(x) \geq \alpha \right\},
\]
and the {\em sendograph of $u$} is defined as $\send(u) := \eend(u) \cap (u_0\times\II)$. Then, for each pair $u,v \in \Fc(X)$ the~{\em endograph metric}~$d_{E}$ on~$\Fc(X)$ is the Hausdorff distance $\com{d}_H$ on $\Cc(X\times\II)$ from $\eend(u)$ to $\eend(v)$, and the {\em sendograph metric} $d_S$ on $\Fc(X)$ is the Hausdorff metric $\com{d}_H$ on $\Kc(X\times\II)$ between the non-empty compact subsets $\send(u)$ and $\send(v)$. We will denote by $\Bc_{E}(u,\eps)$ and $\Bc_{S}(u,\eps)$ the open balls centred at $u \in \Fc(X)$ and of radius $\eps>0$ for each of the metrics $d_{E}$ and $d_{S}$ respectively. We also use the notation $\Fc_{E}(X)$ and $\Fc_{S}(X)$ instead of $(\Fc(X),d_{E})$ and $(\Fc(X),d_{S})$, and the symbols $\tau_{E}$ and $\tau_{S}$ for the respective topologies induced on $\Fc(X)$ by $d_{E}$ and $d_{S}$. Although $d_{E}$ is defined as a distance of closed but not necessarily compact sets on $X\times\II$, the fact that the supports of the fuzzy sets $u,v \in \Fc(X)$ are compact shows that $d_{E}(u,v)$ is finite and hence that $d_{E}$ is indeed a well-defined metric. Moreover, using Proposition~\ref{Pro:Hausdorff} one can check that $d_{E}(u,v) \leq d_{S}(u,v) \leq d_{\infty}(u,v)$ for every $u,v \in \Fc(X)$.

The above metrics come from the general theory of Spaces of Fuzzy Sets, each of them has its own role and importance, and there are many references showing this fact (see \cite{Skorokhod1956_TPA_limit,Billingsley1989_book_convergence,JooKim2000_FSS_the,JardonSanSan2020_FSS_some} for the historical development of the Skorokhod metric, and see \cite{Chen2000_book_fuzzy,Huang2018_FSS_characterizations,GrecoMosReQuel1998_AUFS_on} for the many applications of the endograph and sendograph metrics and their relation with the notion of $\Gamma$-convergence). The study of the relationships between these metrics has been an important topic addressed in the literature, and it is by now well-known that $d_{E}(u,v) \leq d_{S}(u,v) \leq d_{0}(u,v) \leq d_{\infty}(u,v)$ for all $u,v \in \Fc(X)$, being $d_{S}(u,v) \leq d_{0}(u,v)$ the only non-trivial inequality (see \cite[Proposition~2.7]{JardonSan2021_FSS_expansive}). Thus, we have and we will repeatedly use the following inclusion of topologies $\tau_{E} \subset \tau_{S} \subset \tau_{0} \subset \tau_{\infty}$. Moreover, given a continuous map $f:X\longrightarrow X$ on a metric space $(X,d)$, the continuity of the Zadeh extension $\fuz{f}:(\Fc(X),\rho)\longrightarrow(\Fc(X),\rho)$ for every $\rho \in \{ d_{\infty} , d_{0} , d_{S} , d_{E} \}$ is also well-known (see \cite{JardonSanSan2020_FSS_some} and \cite{Kupka2011_IS_on}).

\subsection{A key fact regarding the endograph metric}

Given a property $\Pc$ from those considered in~\cite{AlvarezLoPe2025_FSS_recurrence,BartollMaPeRo2022_AXI_orbit,JardonSanSan2020_MAT_transitivity,MartinezPeRo2021_MAT_chaos}, the way used there to prove the equivalence of such a property~$\Pc$ between the systems $(\Kc(X),\com{f})$, $(\Fc_{\infty}(X),\fuz{f})$ and $(\Fc_{0}(X),\fuz{f})$ is the following:
\begin{enumerate}[--]
	\item \textbf{Step~1}. One shows that if $(\Kc(X),\com{f})$ has property $\Pc$, then $(\Fc_{\infty}(X),\fuz{f})$ has property $\Pc$.
	
	\item \textbf{Step~2}. The inclusion $\tau_{0} \subset \tau_{\infty}$ trivially implies that also $(\Fc_{0}(X),\fuz{f})$ has property $\Pc$.
	
	\item \textbf{Step~3}. One comes back from $(\Fc_{0}(X),\fuz{f})$ to $(\Kc(X),\com{f})$ completing the circle.
\end{enumerate}
The main argument in \textbf{Step~3} is that given any $K \in \Kc(X)$, any $\eps>0$, and a set $u \in \Bc_{0}(\chi_K,\eps)$ fulfilling some kind of $\Pc$-condition with respect to $\fuz{f}$, then for each $\alpha \in \II$ we have that $u_{\alpha} \in \Bc_H(K,\eps)$ and that $u_{\alpha}$ also fulfills the same $\Pc$-condition as $u$ but with respect to $\com{f}$.

For the sendograph and endograph metrics we have that \textbf{Step~2} will follow from the inclusions $\tau_{E} \subset \tau_{S} \subset \tau_{0}$. For \textbf{Step~3}, the following lemma is the key to get \eqref{eq:d_E->d_H}, all the improvements stated in the Introduction, and to solve Questions~\ref{Q:main1}~and~\ref{Q:main2} in a rather general way:

\begin{lemma}\label{Lem:key}
	Let $(X,d)$ be a metric space, and assume that the sets $K \in \Kc(X)$ and $u \in \Fc(X)$ fulfill that $\delta := d_{E}(\chi_K,u) < \tfrac{1}{2}$. Then $d_H(K,u_{\alpha}) \leq \delta$ for every $\alpha \in \ ]\delta,1-\delta]$.
\end{lemma}
\begin{proof}
	Given any arbitrary but fixed $\alpha \in \ ]\delta,1-\delta]$ we must show that $d_H(K,u_{\alpha})\leq\delta$ or, equivalently, that we have the inclusions $K \subset u_{\alpha} + \delta$ and $u_{\alpha} \subset K + \delta$ (see Proposition~\ref{Pro:Hausdorff}).
	
	We start by showing $K \subset u_{\alpha} + \delta$. Indeed, since $d_{E}(\chi_K,u) = \delta$, by Proposition~\ref{Pro:Hausdorff} we have that
	\[
	\eend(\chi_K) \subset \eend(u) + \delta.
	\]
	Hence, given $x \in K$ we have that $(x,1) \in \eend(\chi_K)$ and that there exists $(y,\beta) \in \eend(u)$ for which
	\[
	\com{d}\left( (x,1) , (y,\beta) \right) = \max\{ d(x,y) , |1-\beta| \} \leq \delta.
	\]
	The latter implies that $u(y) \geq \beta \geq 1-\delta \geq \alpha$ so that $y \in u_{\alpha}$. We deduce that $x \in \{y\} + \delta \subset u_{\alpha} + \delta$, and by the arbitrariness of $x \in K$ we get that $K \subset u_{\alpha} + \delta$ as we had to show.
	
	Let us now check that $u_{\alpha} \subset K + \delta$. Indeed, and again since $d_{E}(\chi_K,u) = \delta$, we also have that
	\[
	\eend(u) \subset \eend(\chi_K) + \delta.
	\]
	Hence, given $x \in u_{\alpha}$ there exists some $(y,\beta) \in \eend(\chi_K)$ such that
	\[
	\com{d}\left( (x,\alpha) , (y,\beta) \right) = \max\{ d(x,y) , |\alpha - \beta| \} \leq \delta.
	\]
	Then $\beta \in [\alpha-\delta,\alpha+\delta] \subset \ ]0,1]$, and since $\chi_K(y)\geq\beta>0$ we have that $\chi_K(y)=1$ and $y \in K$. We deduce that $x \in \{y\} + \delta \subset K + \delta$, and by the arbitrariness of $x \in u_{\alpha}$ we get that $u_{\alpha} \subset K + \delta$.
\end{proof}


\section{Transitivity, recurrence and Furstenberg families}\label{Sec_3:t.r.Ff}

In this section we study both transitivity and recurrence-kind properties from the Furstenberg families point of view. We start by improving the main transitivity-kind results from~\cite{JardonSanSan2020_MAT_transitivity} and~\cite{MartinezPeRo2021_MAT_chaos}, solving both~Questions~\ref{Q:main1} and \ref{Q:main2} in the positive (see Remark~\ref{Rem:Q.1.2} below), and then we improve the main recurrence-kind results from~\cite{AlvarezLoPe2025_FSS_recurrence}. Finally we consider the case of complete metric spaces to improve the point-wise results form \cite{AlvarezLoPe2025_FSS_recurrence,JardonSanSan2020_MAT_transitivity}, but also obtaining new outcomes in terms of point-$\Ac$-transitivity.

From now on, given any dynamical system $(X,f)$ and any positive integer $N \in \NN$ we will denote by $f_{(N)}:X^N\longrightarrow X^N$ the {\em $N$-fold direct product} of $f$ with itself, i.e.\ the pair $(X^N,f_{(N)})$ will be the dynamical system
\[
f_{(N)} := \underbrace{f\times\cdots\times f}_{N} : \underbrace{X\times\cdots\times X}_{N} \longrightarrow \underbrace{X\times\cdots\times X}_{N},
\]
where $X^N:=X\times\cdots\times X$ is the {\em $N$-fold direct product} of $X$ with the usual product topology, and where we have the evaluation $f_{(N)}\left((x_1,...,x_N)\right) := (f(x_1),...,f(x_N))$ for each $N$-tuple $(x_1,...,x_N) \in X^N$.

\subsection{Transitivity-kind properties}\label{SubSec_3.1:Transitivity}

Following~\cite{GrossePe2011_book_linear}, the main transitivity-kind notions studied in Topological Dynamics are transitivity itself, and the mixing and weak-mixing properties. We recall that a dynamical system $(X,f)$ is called:
\begin{enumerate}[--]
	\item {\em topologically transitive} (or just {\em transitive}) if given any pair of non-empty open subsets $U,V \subset X$ there exists some (and hence infinitely many) $n \in \NN$ such that $f^n(U) \cap V \neq \varnothing$;
	
	\item {\em topologically mixing} (or just {\em mixing}) if given any pair of non-empty open subsets $U,V \subset X$ there exists some $n_0 \in \NN$ such that $f^n(U) \cap V \neq \varnothing$ for every $n\geq n_0$;
	
	\item {\em topologically weakly-mixing} (or just {\em weakly-mixing}) if the $2$-fold direct product system $(X^2,f_{(2)})$ is topologically transitive.
\end{enumerate}

The notion of {\em mixing} implies that of {\em weak-mixing}, which in its turn implies that of {\em transitivity}. Also, by a well-known result of Furstenberg we have that, once $(X,f)$ is weakly-mixing, then $(X^N,f_{(N)})$ is transitive for all $N \in \NN$ (see \cite[Theorem~1.51]{GrossePe2011_book_linear}). In addition, if $(X,f)$ is weakly-mixing, then either $X$ is a singleton (a fixed point), or $(X,d)$ has no isolated points (see \cite[Exercise~1.2.3]{GrossePe2011_book_linear}). These properties have recently been studied from the point of view of Furstenberg families and $\Ac$-transitivity:
\begin{enumerate}[--]
	\item a collection of sets $\Ac \subset \Part(\NN_0)$ is called a {\em Furstenberg family} (or just a {\em family}) if it is hereditarily upward but $\varnothing \notin \Ac$, i.e.\ if given $A \in \Ac$ and $B \subset \NN_0$ with $A \subset B$ then $B \in \Ac$ but $\Ac \neq \Part(\NN_0)$;
	
	\item a dynamical system $(X,f)$ is called {\em topologically $\Ac$-transitive} (or just {\em $\Ac$-transitive}), for a family $\Ac \subset \Part(\NN_0)$, if given any pair of non-empty open subsets $U,V \subset X$ the {\em return set from $U$ to $V$}, which will be denoted by $\Nc_f(U,V) := \{ n \in \NN_0 \ ; \ f^n(U) \cap V \neq \varnothing \}$, belongs to the family $\Ac$.
\end{enumerate}
Note that: the notion of {\em transitivity} coincides with that of {\em $\Ac_{\infty}$-transitivity} for the family~$\Ac_{\infty}$ formed by the infinite subsets of $\NN_0$; the notion of {\em mixing} coincides with that of {\em $\Ac_{cf}$-transitivity} for the family~$\Ac_{cf}$ of co-finite sets; and it is well-known but non-trivial that the property of {\em weak-mixing} coincides with that of {\em $\Ac_{th}$-transitivity} for the family~$\Ac_{th}$ formed by the so-called {\em thick sets} (i.e.\ the sets $A \subset \NN_0$ such that for every $\ell \in \NN$ there is $n \in A$ with $[\![ n , n+\ell ]\!] \subset A$), see for instance \cite[Theorem~1.54]{GrossePe2011_book_linear}.

We can now address our first main result, which improves \cite[Theorem~3 and Proposition~7]{JardonSanSan2020_MAT_transitivity} since we add the equivalence regarding the system $(\Fc_{E}(X),\fuz{f})$, but also improves \cite[Theorem~2]{MartinezPeRo2021_MAT_chaos} since we only assume that $\Ac$ is any Furstenberg family and not necessarily a filter (see Remark~\ref{Rem:Q.1.2} below):

\begin{theorem}\label{The:A-tran}
	Let $\Ac \subset \Part(\NN_0)$ be a Furstenberg family. Then, given a continuous map $f:X\longrightarrow X$ acting on a metric space $(X,d)$, the following statements are equivalent:
	\begin{enumerate}[{\em(i)}]
		\item $(X,f)$ is weakly-mixing and topologically $\Ac$-transitive;
		
		\item $(X^N,f_{(N)})$ is topologically $\Ac$-transitive for every $N \in \NN$;
		
		\item $(\Kc(X),\com{f})$ is topologically $\Ac$-transitive;
		
		\item $(\Fc_{\infty}(X),\fuz{f})$ is topologically $\Ac$-transitive;
		
		\item $(\Fc_{0}(X),\fuz{f})$ is topologically $\Ac$-transitive;
		
		\item $(\Fc_{S}(X),\fuz{f})$ is topologically $\Ac$-transitive;
		
		\item $(\Fc_{E}(X),\fuz{f})$ is topologically $\Ac$-transitive.
	\end{enumerate}
\end{theorem}
Along the proof of Theorem~\ref{The:A-tran} we repeatedly use that: \textit{any of the statements} (i)\textit{,} (ii) \textit{and} (iii) \textit{implies that both systems $(X,f)$ and $(\Kc(X),\com{f})$ are weakly-mixing}. Indeed, since for every Furstenberg family $\Ac \subset \Part(\NN_0)$ the notion of topological $\Ac$-transitivity implies that of usual transitivity, the previous assertion follows from the (independently obtained and well-known) main result of~\cite{Banks2005_CSF_chaos,LiaoWangZhang2006_SCSAM_transitivity,Peris2005_CSF_set-valued}.
\begin{proof}[Proof of Theorem~\ref{The:A-tran}]
	(i) $\Leftrightarrow$ (ii) $\Leftrightarrow$ (iii): This was proved in~\cite[Lemma~1 and Theorem~3]{FuXing2012_CSF_mixing} for the particular case in which $\Ac$ is a {\em full family} (i.e.\ when the intersection of each set from~$\Ac$ with each set from the dual family~$\Ac^*:=\{ B \subset \NN_0 \ ; \ A \cap B \neq \varnothing \text{ for all } A \in \Ac \}$ is an infinite set). Hence, to check the equivalence for arbitrary Furstenberg families it is enough to show the following fact:
	\begin{enumerate}[--]
		\item \textit{If any of the statements} (i)\textit{,} (ii) \textit{or} (iii) \textit{holds for the family $\Ac \subset \Part(\NN_0)$, then there exists a full family $\Ac' \subset \Part(\NN_0)$ fulfilling the inclusion $\Ac' \subset \Ac$ but also that} (i) \textit{holds for the family $\Ac'$}.
	\end{enumerate}
	To check this fact note that, since any of the statements (i), (ii) and (iii) implies that the system $(X,f)$ is weakly-mixing, then we have two possibilities as mentioned at the beginning of Subsection~\ref{SubSec_3.1:Transitivity}: either $X$ is a singleton (a fixed point) and in this case we can consider the family of co-finite sets $\Ac':=\Ac_{cf}$ and whose dual family is $\Ac_{\infty}$; or else $(X,d)$ has no isolated points (see \cite[Exercise~1.2.3]{GrossePe2011_book_linear}). In this last case, in which $(X,d)$ has no isolated points, we can consider the family
	\[
	\Ac' := \{ A \subset \NN_0 \ ; \ \Nc_f(U,V) \subset A \text{ for some pair of non-empty open subsets } U,V \subset X  \}.
	\]
	It is not hard to check that $\Ac' \subset \Ac$ but also that statement (i) holds for the family $\Ac'$ whenever any of (i), (ii) or (iii) holds for $\Ac$. In order to check that $\Ac'$ is a full family we argue by contradiction: if given any pair of non-empty open subsets $U,V \subset X$ we had that $\Nc_f(U,V) \cap B$ is a finite set for some element of the dual family $B \in (\Ac')^*$, then we could consider the integer
	\[
	m := \max\left( \Nc_f(U,V) \cap B \right) < \infty.
	\]
	Picking a point $x \in U$ and using the continuity of $f$ together with the fact that $X$ (and hence $V$) has no isolated points, we could find a positive value $\delta>0$ small enough to fulfill that the set
	\[
	W := V \setminus \left( \bigcup_{0\leq n\leq m} f^n(\Bc_d(x,\delta)) \right)
	\]
	has non-empty interior, but also that $\Bc_d(x,\delta) := \{ y \in X \ ; \ d(x,y)<\delta \} \subset U$. Setting $U':=\Bc_d(x,\delta)$ and denoting by $V'$ the interior of $W$, we would have two open subsets $U' \subset U$ and $V' \subset V$ for which
	\[
	\Nc_f(U',V') \subset \Nc_f(U,V) \cap \{m+1,m+2,...\}.
	\]
	This is a contradiction: we would have that $\Nc_f(U',V')$ belongs to $\Ac'$ by definition and that $B$ belongs to the dual family $(\Ac')^*$, but $\Nc_f(U',V') \cap B = \varnothing$, which contradicts the definition of dual family.
	
	(iii) $\Rightarrow$ (iv): This was proved in \cite[Theorem~2]{MartinezPeRo2021_MAT_chaos} for the particular case in which the family~$\Ac$ is a filter (see Remark~\ref{Rem:Q.1.2} below). We argue here the initial part of the proof for a (not necessarily filter) family $\Ac$: given arbitrary sets $u,v \in \Fc(X)$ and $\eps>0$ we must show that the return set
	\[
	A := \Nc_{\fuz{f}}\left( \Bc_{\infty}(u,\eps) , \Bc_{\infty}(v,\eps) \right)
	\]
	belongs to $\Ac$. By Lemma~\ref{Lem:eps.pisos} there exist numbers $0 = \alpha_0 < \alpha_1 < \alpha_2 < ... < \alpha_N = 1$ such that
	\[
	\max\left\{ d_H(u_{\alpha}, u_{\alpha_{j+1}}) , d_H(v_{\alpha}, v_{\alpha_{j+1}}) \right\} < \tfrac{\eps}{2} \quad \text{ for each } \alpha \in \ ]\alpha_j,\alpha_{j+1}] \text{ with } 0\leq j\leq N-1.
	\]
	Since $(\Kc(X),\com{f})$ is topologically $\Ac$-transitive and hence weakly-mixing, using now the already proved equivalence (i) $\Leftrightarrow$ (ii) of this Theorem~\ref{The:A-tran} we obtain that $(\Kc(X)^N,\com{f}_{(N)})$ is topologically $\Ac$-transitive for every $N \in \NN$. This last fact implies the following:
	\[
	B := \bigcap_{1\leq j\leq N} \Nc_{\com{f}}\left( \Bc_H(u_{\alpha_j},\tfrac{\eps}{2}) , \Bc_H(v_{\alpha_j},\tfrac{\eps}{2}) \right) \in \Ac.
	\]
	From here the proof follows as in \cite[Theorem~2]{MartinezPeRo2021_MAT_chaos} by showing that $B \subset A$. Indeed, given any arbitrary but fixed $n \in A$ one can find compact sets $K_j \in \Bc_H(u_{\alpha_j},\tfrac{\eps}{2})$ such that $\com{f}^n(K_j) \in \Bc_H(v_{\alpha_j},\tfrac{\eps}{2})$ for each $1\leq j\leq N$, and following the proof of \cite[Theorem~2]{MartinezPeRo2021_MAT_chaos} it can be checked that the fuzzy set $w := \max_{1\leq j\leq N} (\alpha_j \cdot K_j) \in \Fc(X)$ belongs to $\Bc_{\infty}(u,\eps)$ but also that $\fuz{f}^n(w) \in \Bc_{\infty}(v,\eps)$.
		
	(iv) $\Rightarrow$ (v) $\Rightarrow$ (vi) $\Rightarrow$ (vii): This trivially follows from the inclusions $\tau_{E} \subset \tau_{S} \subset \tau_{0} \subset \tau_{\infty}$.
	
	(vii) $\Rightarrow$ (iii): Given arbitrary sets $K,L \in \Kc(X)$ and $\eps>0$, we must show that the return set
	\[
	A := \Nc_{\com{f}}(\Bc_H(K,\eps),\Bc_H(L,\eps))
	\]
	belongs to $\Ac$. Without loss of generality we will assume that $0<\eps\leq\frac{1}{2}$. Now, since the dynamical system $(\Fc_{E}(X),\fuz{f})$ is assumed to be topologically $\Ac$-transitive we have that
	\[
	B := \Nc_{\fuz{f}}\left( \Bc_{E}(\chi_K,\eps) , \Bc_{E}(\chi_L,\eps) \right) \in \Ac.
	\]
	Thus, given any $n \in B$ there exists some $u \in \Bc_{E}(\chi_K,\eps)$ such that $\fuz{f}^n(u) \in \Bc_{E}(\chi_L,\eps)$. Letting now
	\[
	\delta := \max\left\{ d_{E}\left( \chi_K , u \right) , d_{E}( \chi_L , \fuz{f}^n(u) ) \right\} < \eps \leq \tfrac{1}{2},
	\]
	and picking any $\alpha \in \ ]\delta,1-\delta]$, an application of Lemma~\ref{Lem:key} shows that $d_H(K,u_{\alpha})<\eps$ and that
	\[
	d_H( L , \com{f}^n(u_{\alpha}) ) = d_H ( L , [\fuz{f}^n(u)]_{\alpha} ) < \eps.
	\]
	Hence $u_{\alpha} \in \Bc_H(K,\eps)$ and $\com{f}^n(u_{\alpha}) \in \Bc_H(L,\eps)$. We conclude that $n \in A$. The arbitrariness of $n \in B$ shows that $B \subset A$ and finally that $A \in \Ac$, as we had to show.
\end{proof}

\begin{remark}\label{Rem:Q.1.2}
	The reader should note that:
	\begin{enumerate}[(1)]
		\item A positive answer to both Questions~\ref{Q:main1} and \ref{Q:main2} follows now from Theorem~\ref{The:A-tran} if we consider the family~$\Ac_{\infty}$ formed by the infinite subsets of $\NN_0$. As a consequence we have that Theorem~\ref{The:A-tran} significantly improves \cite[Theorem~3]{JardonSanSan2020_MAT_transitivity} but also \cite[Proposition~7]{JardonSanSan2020_MAT_transitivity}. Alternatively:
		
		\item The fact that ``{\em if $(\Fc_{E}(X),\fuz{f})$ is transitive, then $(X,f)$ is weakly-mixing}'' can be directly proved using \cite[Proposition~1.53]{GrossePe2011_book_linear}, which states that: {\em a dynamical system $(X,f)$ is weakly-mixing if,~and only if, for any pair of non-empty open subsets $U,V \subset X$ we have that $\Nc_f(U,U) \cap \Nc_f(U,V) \neq \varnothing$}. Indeed, to check that $(X,f)$ is weakly-mixing let $x,y \in X$ and $\eps>0$ be arbitrary but fixed, and let us prove that~$\Nc_f(\Bc_d(x,\eps),\Bc_d(x,\eps)) \cap \Nc_f(\Bc_d(x,\eps),\Bc_d(y,\eps)) \neq \varnothing$, where $\Bc_d(x,\eps)$ and $\Bc_d(y,\eps)$ are the open $d$-balls of radius $\eps$ centered at $x$ and $y$ respectively. Assuming that $(\Fc_{E}(X),\fuz{f})$ is transitive there are $u \in \Bc_{E}(\chi_{\{x\}},\eps)$ and $n \in \NN$ such that $\fuz{f}^n(u) \in \Bc_{E}(\chi_{\{x,y\}},\eps)$. Finally, if $\eps$ is small enough one can apply Lemma~\ref{Lem:key} exactly as in Theorem~\ref{The:A-tran} to obtain the existence of points $z_1, z_2 \in \Bc_d(x,\eps)$ fulfilling that $f^n(z_1) \in \Bc_d(x,\eps)$ and $f^n(z_2) \in \Bc_d(y,\eps)$, which completes the proof.
		
		\item In statement (i) of Theorem~\ref{The:A-tran} we can not drop the assumption ``{\em $(X,f)$ is weakly-mixing}'' unless the family $\Ac$ is a {\em filter}, i.e.\ unless we have the condition $A \cap B \in \Ac$ for every $A,B \in \Ac$. Indeed, the following example is well-known in Topological Dynamics:
		\begin{enumerate}[--]
			\item Let $\alpha \in \RR\setminus\QQ$. The system $(\TT,f_{\alpha})$, defined on $\TT:=\{ z \in \CC \ ; \ |z|=1 \}$ as $f_{\alpha}(z):=ze^{i\pi\alpha}$, is transitive and even {\em $\Ac_{syn}$-transitive} for the family $\Ac_{syn}$ formed by the {\em syndetic sets} (i.e.\ the sets $A \subset \NN_0$ admitting some $\ell \in \NN$ such that $A \cap [\![ n , n+\ell ]\!] \neq \varnothing$ for all $n \in \NN_0$). However, it is not hard to check that $(\TT,f_{\alpha})$ is not weakly-mixing (see \cite[Example~1.43]{GrossePe2011_book_linear}).
		\end{enumerate}
		This example directly shows, although it is not hard to check, that the family $\Ac_{syn}$ is not a filter.
		
		Let us also mention that once a system $(X,f)$ is weakly-mixing and $\Ac$-transitive then we have that the respective family~$\Ac'$, considered in the proof of Theorem~\ref{The:A-tran} and defined as
		\[
		\Ac' := \{ A \subset \NN_0 \ ; \ \Nc_f(U,V) \subset A \text{ for some pair of non-empty open subsets } U,V \subset X  \},
		\]
		is a subfamily of $\Ac$ fulfilling that $\Ac'$ is a filter (such a filter condition follows from \cite[Lemma~1.50]{GrossePe2011_book_linear} applied to the weakly-mixing property). Hence, if a dynamical system $(X,f)$ is weakly-mixing and $\Ac$-transitive, then the system $(X,f)$ is indeed $\Ac'$-transitive for the filter $\Ac'$. This reasoning shows that Theorem~\ref{The:A-tran} can only be considered as a slight improvement of~\cite[Theorem~2]{MartinezPeRo2021_MAT_chaos}, which statement requires the filter condition of the Furstenberg family under consideration.
	\end{enumerate}
\end{remark}

\subsection{Recurrence-kind properties}

Following \cite{AlvarezLoPe2025_FSS_recurrence}, and having already introduced the concept of Furstenberg family, let us recall the recurrence-kind properties that we consider in this paper. A dynamical system $(X,f)$ is called:
\begin{enumerate}[--]
	\item {\em topologically recurrent} (or just {\em recurrent}) if given any non-empty open subset $U \subset X$ there exists some (and hence infinitely many) $n \in \NN$ such that $f^n(U) \cap U \neq \varnothing$;
	
	\item {\em topologically multiply recurrent} (or just {\em multiply recurrent}) if given any positive integer $\ell \in \NN$ and any non-empty open subset $U \subset X$ there exists some (and hence infinitely many) $n \in \NN$ such that
	\[
	\bigcap_{0\leq j\leq \ell} f^{-jn}(U) = U \cap f^{-n}(U) \cap f^{-2n}(U) \cap \cdots \cap f^{-\ell n}(U) \neq \varnothing;
	\]
	
	\item {\em topologically $\Ac$-recurrent} (or just {\em $\Ac$-recurrent}), for a Furstenberg family $\Ac \subset \Part(\NN_0)$, if given any non-empty open subset $U \subset X$ we have that the return set $\Nc_f(U,U) = \{ n \in \NN_0 \ ; \ f^n(U) \cap U \neq \varnothing \}$, which from now on will be simply denoted by $\Nc_f(U) := \Nc_f(U,U)$, belongs to the family $\Ac$.
\end{enumerate}
These properties have recently been studied (see for instance \cite{CostakisMaPa2014_CAOT_recurrent,KwietniakLiOYe2017_SCM_multi-recurrence} and \cite{AbakumovA2026_CJM_on}). Note that, even though the notion of {\em topological recurrence} is implied by that of {\em multiple recurrence} and also by that of {\em topological $\Ac$-recurrence} for any family $\Ac$, there is no clear relation between {\em multiple recurrence} and {\em topological $\Ac$-recurrence} as argued in \cite[Remark~2.1]{AlvarezLoPe2025_FSS_recurrence}. For that reason, and in order to give a unified treatment to the previous properties, we will use the following notion recently introduced in~\cite{AlvarezLoPe2025_FSS_recurrence}:
\begin{enumerate}[--]
	\item a dynamical system $(X,f)$ is called {\em topologically $(\ell,\Ac)$-recurrent} (or just {\em $(\ell,\Ac)$-recurrent}), for a positive integer $\ell \in \NN$ and a Furstenberg family $\Ac \subset \Part(\NN)$, if given any non-empty open subset $U \subset X$ we have that the {\em $\ell$-return set from $U$ to itself}, which will be denoted by 
	\[
	\Nc_f^{\ell}(U) := \left\{ n \in \NN_0 \ ; \ \bigcap_{0\leq j\leq \ell} f^{-jn}(U) \neq \varnothing \right\},
	\]
	belongs to the family $\Ac$.
\end{enumerate}
Our next results improves \cite[Theorem~4.1]{AlvarezLoPe2025_FSS_recurrence} by adding the sendograph and endograph equivalences:

\begin{theorem}\label{The:(l,A)-rec}
	Let $\ell \in \NN$ be a positive integer and let $\Ac \subset \Part(\NN_0)$ be a Furstenberg family. Then, given a continuous map $f:X\longrightarrow X$ on a metric space $(X,d)$, the following statements are equivalent:
	\begin{enumerate}[{\em(i)}]		
		\item $(X^N,f_{(N)})$ is topologically $(\ell,\Ac)$-recurrent for every $N \in \NN$;
		
		\item $(\Kc(X),\com{f})$ is topologically $(\ell,\Ac)$-recurrent;
		
		\item $(\Fc_{\infty}(X),\fuz{f})$ is topologically $(\ell,\Ac)$-recurrent;
		
		\item $(\Fc_{0}(X),\fuz{f})$ is topologically $(\ell,\Ac)$-recurrent;
		
		\item $(\Fc_{S}(X),\fuz{f})$ is topologically $(\ell,\Ac)$-recurrent;
		
		\item $(\Fc_{E}(X),\fuz{f})$ is topologically $(\ell,\Ac)$-recurrent.
	\end{enumerate}
	As a consequence, these equivalences hold if we replace ``topological $(\ell,\Ac)$-recurrence'' with any of the particular properties ``topological recurrence'', ``multiple recurrence'' or ``topological $\Ac$-recurrence''.
\end{theorem}
Note that the last sentence of Theorem~\ref{The:(l,A)-rec} follows because: {\em topological recurrence} coincides with {\em topological $(1,\Ac_{\infty})$-recurrence}; {\em multiple recurrence} coincides with being {\em topologically $(\ell,\Ac_{\infty})$-recurrent} for every $\ell \in \NN$; and {\em topological $\Ac$-recurrence} coincides with {\em topological $(1,\Ac)$-recurrence}.
\begin{proof}[Proof of Theorem~\ref{The:(l,A)-rec}]
	(i) $\Leftrightarrow$ (ii) $\Leftrightarrow$ (iii) $\Leftrightarrow$ (iv): This was proved in~\cite[Theorem~4.1]{AlvarezLoPe2025_FSS_recurrence}.
	
	(iv) $\Rightarrow$ (v) $\Rightarrow$ (vi): This trivially follows from the inclusions $\tau_{E} \subset \tau_{S} \subset \tau_{0}$.
	
	(vi) $\Rightarrow$ (ii): Given any set $K \in \Kc(X)$ and any $\eps>0$ we must show that the return set
	\[
	A := \Nc_{\com{f}}^{\ell}( \Bc_H(K,\eps) )
	\]
	belongs to $\Ac$. Without loss of generality we will assume that $0<\eps\leq\frac{1}{2}$. Since $(\Fc_{E}(X),\fuz{f})$ is assumed to be topologically $(\ell,\Ac)$-recurrent we have that
	\[
	B := \Nc_{\fuz{f}}^{\ell}\left( \Bc_{E}(\chi_K,\eps) \right) \in \Ac.
	\]
	Thus, given any arbitrary but fixed $n \in B$ there exists $u \in \Fc(X)$ fulfilling that
	\[
	u \in \bigcap_{0\leq j\leq \ell} \fuz{f}^{-jn}\left( \Bc_{E}(\chi_K,\eps) \right).
	\]
	Letting now
	\[
	\delta := \max\left\{ d_{E}( \chi_K , \fuz{f}^{jn}(u) ) \ ; \ 0\leq j\leq \ell \right\} < \eps \leq \tfrac{1}{2},
	\]
	and picking any $\alpha \in \ ]\delta,1-\delta]$, an application of Lemma~\ref{Lem:key} shows that
	\[
	d_H( K , \com{f}^{jn}(u_{\alpha}) ) = d_H( K , [\fuz{f}^{jn}(u)]_{\alpha} ) < \eps \quad \text{ for all } 0\leq j\leq \ell.
	\]
	Hence we have that
	\[
	u_{\alpha} \in \bigcap_{0\leq j\leq \ell} \com{f}^{-jn}( \Bc_H(K,\eps) ). 
	\]
	We conclude that $n \in A$. The arbitrariness of $n \in B$ shows that $B \subset A$ and finally that $A \in \Ac$.
\end{proof}

\begin{remark}
	As it happens in Theorem~\ref{The:A-tran} and we mention in part (3) of Remark~\ref{Rem:Q.1.2}, if the Furstenberg family~$\Ac$ considered in the statement of Theorem~\ref{The:(l,A)-rec} is not a filter then we cannot weaken the assumption ``{\em $(X^N,f_{(N)})$ is topologically $(\ell,\Ac)$-recurrent for every $N \in \NN$}''. Indeed, it has recently been showed in~\cite[Theorem~3.2]{GrivauxLoPe2025_AMP_questions-I} that given any $N \in \NN$ there exists a dynamical system $(X,f)$ fulfilling that $(X^N,f_{(N)})$ is topologically recurrent (and even multiply recurrent), but also fulfilling that the system $(X^{N+1},f_{(N+1)})$ is not topologically recurrent (neither multiply recurrent).
\end{remark}

\subsection{The case of complete metric spaces}

Following~\cite{AlvarezLoPe2025_FSS_recurrence,JardonSanSan2020_MAT_transitivity}, if given a system $(X,f)$ we assume that $(X,d)$ is a {\em complete} metric space, then we can improve Theorems~\ref{The:A-tran} and \ref{The:(l,A)-rec} in terms of {\em point-transitivity} and {\em point-recurrence} by using some concrete {\em Baire category arguments}. In this subsection we also consider the generalizations of these properties in terms of Furstenberg families, namely {\em point-$\Ac$-transitivity} and {\em point-$\Ac$-recurrence}, obtaining some completely new results (see Lemma~\ref{Lem:point-A-rec/tran}, Theorem~\ref{The:point-A-tran} and Corollary~\ref{Cor:point-BD-tran} below).

Let us start by recalling the basic definitions used in \cite{AlvarezLoPe2025_FSS_recurrence,JardonSanSan2020_MAT_transitivity}: given a dynamical system $(X,f)$ and a point $x \in X$ we will denote its {\em $f$-orbit} by
\[
\orb_f(x) := \{ f^n(x) \ ; \ n \in \NN_0 \},
\]
and given any subset $U \subset X$ we will denote the {\em return set from $x$ to $U$} by
\[
\Nc_f(x,U) := \{ n \in \NN_0 \ ; \ f^n(x) \in U \}.
\]
Following \cite{AlvarezLoPe2025_FSS_recurrence} and \cite{JardonSanSan2020_MAT_transitivity} we will say that a point $x \in X$ is:
\begin{enumerate}[--]
	\item {\em transitive for $f$} if the $f$-orbit of $x$ is dense in the space $X$, i.e.\ $X = \cl{\orb_f(x)}$; equivalently, if for every non-empty open subset $U \subset X$ the return set $\Nc_f(x,U)$ is non-empty. We will denote by $\Tran(f)$ the {\em set of transitive points for $f$}, and $(X,f)$ is called {\em point-transitive} if the set $\Tran(f)$ is non-empty.
	
	\item {\em recurrent for $f$} if $x$ belongs to the closure of its forward orbit, i.e.\ $x \in \cl{\orb_f(f(x))}$; equivalently, if for every neighbourhood $U$ of $x$ the return set $\Nc_f(x,U)$ is an infinite set. We will denote by~$\Rec(f)$ the {\em set of recurrent points for $f$}, and $(X,f)$ is called {\em point-recurrent} if the set $\Rec(f)$ is dense in $X$.
	
	\item {\em $\AP$-recurrent for $f$} if for every neighbourhood $U$ of $x$ the return set $\Nc_f(x,U)$ contains arbitrarily long arithmetic progressions, i.e.\ if for each positive integer $\ell \in \NN$ there exist two numbers $n \in \NN$ and $n_0 \in \NN_0$ such that $\{ n_0+jn \ ; \ 0\leq j\leq \ell \} \subset \Nc_f(x,U)$. We will denote by $\AP\Rec(f)$ the {\em set of $\AP$-recurrent points for $f$}, and $(X,f)$ is called {\em point-$\AP$-recurrent} if $\AP\Rec(f)$ is dense in $X$.
\end{enumerate}

Let us recall: firstly, that every {\em point-recurrent} dynamical system is {\em topologically recurrent}, that {\em point-$\AP$-recurrence} implies {\em multiply recurrence}, and that the converse statements hold for completely metrizable spaces (see \cite[Proposition~2.1]{CostakisMaPa2014_CAOT_recurrent} and \cite[Lemma~4.8]{KwietniakLiOYe2017_SCM_multi-recurrence}); and secondly, that {\em point-transitivity} implies {\em topological transitivity} on metric spaces without isolated points, and that the converse holds on separable completely metrizable spaces (see \cite[Theorem~1.16]{GrossePe2011_book_linear}). Moreover, one can show with not too much difficulty that for any metric space $(X,d)$, the associated fuzzy metric space $\Fc_E(X)$ is either a singleton (precisely when $X$ is a singleton) or has no isolated points, regardless of whether $(X,d)$ itself has isolated points (a proof of this fact will appear in \cite{Lopez2026_JIA_topological-II}).

We are now ready to improve both \cite[Corollary~5.1]{AlvarezLoPe2025_FSS_recurrence} and \cite[Theorem~4]{JardonSanSan2020_MAT_transitivity}:

\begin{corollary}\label{Cor:point-rec/tran}
	Let $f:X\longrightarrow X$ be a continuous map on a complete metric space $(X,d)$. Hence:
	\begin{enumerate}[{\em(a)}]
		\item The following statements are equivalent:
		\begin{enumerate}[{\em(i)}]
			\item $(X^N,f_{(N)})$ is point-recurrent (resp.\ point-$\AP$-recurrent) for every $N \in \NN$;
			
			\item $(\Kc(X),\com{f})$ is point-recurrent (resp.\ point-$\AP$-recurrent);
			
			\item $(\Fc_{\infty}(X),\fuz{f})$ is point-recurrent (resp.\ point-$\AP$-recurrent);
			
			\item $(\Fc_{0}(X),\fuz{f})$ is point-recurrent (resp.\ point-$\AP$-recurrent);
			
			\item $(\Fc_{S}(X),\fuz{f})$ is point-recurrent (resp.\ point-$\AP$-recurrent);
			
			\item $(\Fc_{E}(X),\fuz{f})$ is point-recurrent (resp.\ point-$\AP$-recurrent).
		\end{enumerate}
		
		\item If $(X,d)$ is also separable, then the following statements are equivalent:
		\begin{enumerate}[{\em(i)}]
			\item $(X^2,f_{(2)})$ is point-transitive;
			
			\item $(\Kc(X),\com{f})$ is point-transitive;
			
			\item $(\Fc_{\infty}(X),\fuz{f})$ is topologically transitive;
			
			\item $(\Fc_{0}(X),\fuz{f})$ is point-transitive;
			
			\item $(\Fc_{S}(X),\fuz{f})$ is point-transitive;
			
			\item $(\Fc_{E}(X),\fuz{f})$ is point-transitive.
		\end{enumerate}
	\end{enumerate}
\end{corollary}
\begin{proof}
	(i) $\Leftrightarrow$ (ii) $\Leftrightarrow$ (iii) $\Leftrightarrow$ (iv): Part (a) was proved in \cite[Corollary~5.1]{AlvarezLoPe2025_FSS_recurrence} and (b) in \cite[Theorem~4]{JardonSanSan2020_MAT_transitivity}.
	
	(iv) $\Rightarrow$ (v) $\Rightarrow$ (vi): This trivially follows from the inclusions $\tau_{E} \subset \tau_{S} \subset \tau_{0}$.

	(vi) $\Rightarrow$ (i): If $(\Fc_{E}(X),\fuz{f})$ is point-recurrent, point-$\AP$-recurrent or point-transitive respectively, then $(\Fc_{E}(X),\fuz{f})$ is topologically recurrent, multiply recurrent or topologically transitive respectively (for topological transitivity we used that $\Fc_{E}(X)$ is either a singleton or has no isolated points). Thus, Theorem~\ref{The:(l,A)-rec} in case (a) and Theorem~\ref{The:A-tran} in case (b), for the family~$\Ac_{\infty}$ formed by the infinite subsets of $\NN_0$, imply that $(X^N,f_{(N)})$ is topologically recurrent, multiply recurrent or topologically transitive respectively, for every $N \in \NN$. Statement (i) in part (a) follows now from the completeness assumption on~$(X,d)$ by using \cite[Proposition~2.1]{CostakisMaPa2014_CAOT_recurrent} or~\cite[Lemma~4.8]{KwietniakLiOYe2017_SCM_multi-recurrence} respectively. Statement (i) in part (b) follows from the separability and completeness assumptions on $(X,d)$ by using \cite[Theorem~1.16]{GrossePe2011_book_linear}.
\end{proof}

In view of Corollary~\ref{Cor:point-rec/tran} one may wonder if the notions of {\em point-transitivity} and {\em point-recurrence} can also be studied from the Furstenberg families point of view. Indeed, following \cite{BesMePePu2016_MA_recurrence}, given a Furstenberg family $\Ac \subset \Part(\NN_0)$ and a dynamical system $(X,f)$ we will say that a point $x \in X$ is:
\begin{enumerate}[--]
	\item {\em $\Ac$-transitive for $f$} if $\Nc_f(x,U) \in \Ac$ for every non-empty open subset $U \subset X$. We denote by~$\Ac\Tran(f)$ the {\em set of $\Ac$-transitive points for $f$}, and $(X,f)$ is called {\em point-$\Ac$-transitive} if $\Ac\Tran(f)$ is non-empty.
		
	\item {\em $\Ac$-recurrent for $f$} if $\Nc_f(x,U) \in \Ac$ for every neighbourhood $U$ of $x$. We denote by~$\Ac\Rec(f)$ the {\em set of $\Ac$-recurrent points for $f$}, and $(X,f)$ is called {\em point-$\Ac$-recurrent} if $\Ac\Rec(f)$ is dense in $X$.
\end{enumerate}
To finish the section we aim to derive a ``point-wise version'' of Theorem~\ref{The:A-tran}. This kind of question seems to be completely new in the framework of ``{\em comparing the dynamical properties of $(X,f)$ with those of $(\Kc(X),\com{f})$ and $(\Fc(X),\fuz{f})$}''. The main difficulty here will be that for a general Furstenberg family~$\Ac\subset\Part(\NN_0)$ there is no Baire category argument to obtain an $\Ac$-transitive point, even under the assumption of very strong topological transitivity properties such as topological mixing. To avoid this issue we will consider a particular class of families (see Theorem~\ref{The:point-A-tran} below). Let us first prove the following general lemma, whose statement (b) significantly improves \cite[Proposition~9]{JardonSanSan2020_MAT_transitivity}:

\begin{lemma}\label{Lem:point-A-rec/tran}
	Let $\Ac \subset \Part(\NN_0)$ be a Furstenberg family. Then, given a continuous map $f:X\longrightarrow X$ on a (not necessarily separable, neither complete) metric space $(X,d)$, the following statements hold:
	\begin{enumerate}[{\em(a)}]
		\item If $(\Kc(X),\com{f})$ is point-$\Ac$-transitive, then $(X,f)$ is point-$\Ac$-transitive.
		
		\item If $(\Fc_{E}(X),\fuz{f})$ is point-$\Ac$-transitive, then $(X,f)$ is point-$\Ac$-transitive.
		
		\item If $(X,f)$ is point-$\Ac$-transitive, then $(X^N,f_{(N)})$ is point-$\Ac$-recurrent for every $N \in \NN$.
		
		\item If $(X^N,f_{(N)})$ is point-$\Ac$-recurrent for every $N \in \NN$, then $(\Kc(X),\com{f})$ is point-$\Ac$-recurrent.
		
		\item If $(X^N,f_{(N)})$ is point-$\Ac$-recurrent for every $N \in \NN$, then $(\Fc_{\infty}(X),\fuz{f})$ is point-$\Ac$-recurrent.
	\end{enumerate}
\end{lemma}
\begin{proof}
	(a): Let $K \in \Kc(X)$ be an $\Ac$-transitive point for the map $\com{f}$ and pick any point $x \in K$. We claim that $x \in \Ac\Tran(f)$. Indeed, fixed any $y \in X$ and $\eps>0$, since $K \in \Ac\Tran(\com{f})$ we have that
	\[
	\Nc_{\com{f}}(K,\Bc_H(\{y\},\eps)) \in \Ac.
	\]
	Hence, given any $n \in \Nc_{\com{f}}(K,\Bc_H(\{y\},\eps))$ we have that $f^n(x) \in f^n(K) = \com{f}^n(K) \subset \{y\} + \delta_n$ for some $0<\delta_n<\eps$, so that $d(f^n(x),y)\leq\delta_n<\eps$. We deduce that $\Nc_{\com{f}}(K,\Bc_H(\{y\},\eps)) \subset \Nc_f(x,\Bc_d(y,\eps))$, which implies that $\Nc_f(x,\Bc_d(y,\eps)) \in \Ac$. The arbitrariness of $y \in X$ and $\eps>0$ proves the claim.
	
	(b): Let $u \in \Fc(X)$ be an $\Ac$-transitive point for $\fuz{f}$ with respect to $d_{E}$ and pick any point $x \in u_1$. We claim that $x \in \Ac\Tran(f)$. Indeed, fixed any point $y \in X$ and any positive value $0<\eps\leq1$, since $u$ is $\Ac$-transitive we have that
	\[
	\Nc_{\fuz{f}}(u,\Bc_{E}(\chi_{\{y\}},\eps)) \in \Ac.
	\]
	Hence, given any arbitrary but fixed $n \in \Nc_{\fuz{f}}(u,\Bc_{E}(\chi_{\{y\}},\eps))$ we have that $\delta_n := d_{E}( \fuz{f}^n(u) , \chi_{\{y\}} ) < \eps$. The latter implies that $\eend(\fuz{f}^n(u)) \subset \eend(\chi_{\{y\}}) + \delta_n$. In particular, since $[ \fuz{f}^n(u) ]_1 = f^n(u_1)$, we have that $(f^n(x),1) \in \eend(\fuz{f}^n(u))$ and hence there exist some $(z,\alpha) \in \eend(\chi_{\{y\}})$ such that
	\[
	\com{d}_H\left( (f^n(x),1) , (z,\alpha) \right) := \max\left\{ d(f^n(x),z) , |1-\alpha| \right\} \leq \delta_n < \eps \leq 1.
	\]
	We deduce that $\chi_{\{y\}}(z)\geq\alpha>0$, so that $z=y$, $d(f^n(x),y)<\eps$ and hence $n \in \Nc_f(x,\Bc_d(y,\eps))$. The arbitrariness of $n \in \Nc_{\fuz{f}}(u,\Bc_{E}(\chi_{\{y\}},\eps))$ shows that $\Nc_{\fuz{f}}(u,\Bc_{E}(\chi_{\{y\}},\eps)) \subset \Nc_f(x,\Bc_d(y,\eps))$ and hence that $\Nc_f(x,\Bc_d(y,\eps)) \in \Ac$. The arbitrariness of $y \in X$ and $0<\eps\leq 1$ proves the claim.
	
	(c): This is a well-known general dynamical systems property that has recently and independently been proved in \cite[Proposition~3.8]{CarMur2025_RACSAM_frequently} and \cite[Proposition~4.3 and Remark~4.4]{GrivauxLoPe2025_BJMA_questions-II}. Let us mention that in the references \cite{CarMur2025_RACSAM_frequently,GrivauxLoPe2025_BJMA_questions-II} the notion of ``point-$\Ac$-transitivity'' is called ``$\Ac$-hypercyclicity'' and the notion of ``point-$\Ac$-recurrence'' is simply called ``$\Ac$-recurrence'' for each Furstenberg family $\Ac \subset \Part(\NN_0)$.
	
	To prove statements (d) and (e) we will use that: \textit{for any $K \in \Kc(X)$ and $\eps>0$ there exist points $y_1,...,y_N \in K$ such that every finite set $L := \{ z_1, z_2, ..., z_N \}$ with $d(y_j,z_j)<\frac{\eps}{2}$ for all $1\leq j\leq N$ fulfills that $d_H(K,L)<\eps$}. This fact is probably folklore, but the reader can easily check it by using the compactness condition of $K$ together with statement (a) of Proposition~\ref{Pro:Hausdorff}.
	
	(d): Given $K \in \Kc(X)$ and $\eps>0$ let $y_1,...,y_N \in K$ be such that every $L := \{ z_1, z_2, ..., z_N \}$ with $d(y_j,z_j)<\frac{\eps}{2}$ for all $1\leq j\leq N$ fulfills that $d_H(K,L)<\eps$. Using the point-$\Ac$-recurrence of $(X^N,f_{(N)})$ we can choose a point
	\[
	(x_1,...,x_N) \in \Bc_d(y_1,\tfrac{\eps}{2})\times\cdots\times \Bc_d(y_N,\tfrac{\eps}{2}) \cap \Ac\Rec(f_{(N)}).
	\]
	We omit the routine verification that $\{x_1,...,x_N\} \in \Bc_H(K,\eps) \cap \Ac\Rec(\com{f})$, which implies that the set $\Ac\Rec(\com{f})$ is dense in $\Kc(X)$ by the arbitrariness of $K \in \Kc(X)$ and $\eps>0$.
	
	(e): Given any normal fuzzy set $u \in \Fc(X)$ and $\eps>0$ we can use Lemma~\ref{Lem:eps.pisos} to obtain positive numbers $0 = \alpha_0 < \alpha_1 < \alpha_2 < ... < \alpha_N = 1$ such that $d_H(u_{\alpha}, u_{\alpha_{j+1}})<\tfrac{\eps}{2}$ for each $\alpha \in \ ]\alpha_j, \alpha_{j+1}]$ with $0\leq j\leq N-1$. Hence, for each $1\leq j\leq N$, one can find points $y_1^j,...,y_{k_j}^j \in u_{\alpha_j}$ such that every finite set $L := \{ z_1, z_2, ..., z_{k_j} \}$ with $d(y_l^j,z_l)<\frac{\eps}{4}$ for all $1\leq l\leq k_j$ fulfills that $d_H(u_{\alpha_j},L)<\tfrac{\eps}{2}$. Using the point-$\Ac$-recurrence of $(X^M,f_{(M)})$ for $M=\sum_{j=1}^N k_j$ we can choose a point
	\[
	(x_1^1,...,x_{k_1}^1,...,x_1^j,...,x_{k_j}^j,...,x_1^N,...,x_{k_N}^N) \in \prod_{l=1}^{k_1} \Bc_d(y_l^1,\tfrac{\eps}{4}) \times \cdots \times \prod_{l=1}^{k_j} \Bc_d(y_l^j,\tfrac{\eps}{4}) \times \cdots \times \prod_{l=1}^{k_N} \Bc_d(y_l^N,\tfrac{\eps}{4})
	\]
	but also fulfilling that
	\[
	(x_1^1,...,x_{k_1}^1,...,x_1^j,...,x_{k_j}^j,...,x_1^N,...,x_{k_N}^N) \cap \Ac\Rec(f_{(M)}).
	\]
	Considering $K_j := \{ x_1^j,...,x_{k_j}^j \}$ for each $1\leq j\leq N$ and letting $v:=\max_{1\leq j\leq N} (\alpha_j \cdot \chi_{K_j})$ it follows that $v$ is a normal fuzzy set. Arguing now as in \cite[Theorem~4.1~and~Remark~5.2]{AlvarezLoPe2025_FSS_recurrence}, which essentially uses statement (b) of Proposition~\ref{Pro:Hausdorff}, one can check that $v \in \Bc_{\infty}(u,\eps) \cap \Ac\Rec(\fuz{f})$. This finally implies that the set $\Ac\Rec(\fuz{f})$ is dense in $\Fc_{\infty}(X)$ by the arbitrariness of $u \in \Fc(X)$ and $\eps>0$.
\end{proof}

We are about to obtain a ``point-wise version'' of Theorem~\ref{The:A-tran} but with a slight restriction on the Furstenberg families allowed. In particular, following \cite[Definition~3.3]{CarMur2025_RACSAM_frequently} but also \cite{Glasner2004_TMNA_classifying,HuangLiYe2012_TAMS_family,Li2011_TA_transitive}:
\begin{enumerate}[--]
	\item given a Furstenberg family $\Ac \subset \Part(\NN_0)$ we define $b\Ac$, called the {\em block family associated to~$\Ac$}, in the following way: a set $B \subset \NN_0$ belongs to $b\Ac$ if there exists some $A_B \in \Ac$ such that for each finite subset $F \subset A_B$ there is some $n_F \in \NN_0$ for which $F + n_F := \{ m+n_F \ ; \ m \in F \} \subset B$.
	
	\item a Furstenberg family $\Ac \subset \Part(\NN_0)$ is called a {\em block family} if we have the equality $b\Ac = \Ac$.
\end{enumerate}
Roughly speaking, the block family $b\Ac$ associated to $\Ac$ is the collection of sets that contain every finite block $F$ from a fixed set of the original family~$\Ac$, but possibly translated by $n_F$. Some examples and basic properties, such as the inclusion $\Ac \subset b\Ac$ or the equality $b(b\Ac)=b\Ac$ for every family $\Ac \subset \Part(\NN_0)$, are exposed in \cite[Section~3]{CarMur2025_RACSAM_frequently}. The next crucial fact is also essentially proved there:
\begin{enumerate}[--]
	\item \cite[Proposition~3.6]{CarMur2025_RACSAM_frequently}: \textit{Let $\Ac \subset \Part(\NN_0)$ be a block Furstenberg family. If a dynamical system $(X,f)$ is both point-transitive and point-$\Ac$-recurrent, then $(X,f)$ is point-$\Ac$-transitive.}
\end{enumerate}
For the shake of completeness, and since this result was originally proved for linear dynamical systems, we include here the main argument of its proof: if $(X,f)$ is both point-transitive and point-$\Ac$-recurrent, then we claim that every transitive point is indeed $\Ac$-transitive. In fact, given any point $x \in \Tran(f)$ and any non-empty open subset $U \subset X$ we can pick an $\Ac$-recurrent point $y \in U \cap \Ac\Rec(f)$ so that $\Nc_f(y,U) \in \Ac$. Fixed now any finite subset $F \subset \Nc_f(y,U)$ we have that $V_F := \bigcap_{m \in F} f^{-m}(U)$ is an open subset of $X$ containing $y$, so that there exists $n_F \in \NN_0$ for which $f^{n_F}(x) \in V_F$ and hence $F+n_F \subset \Nc_f(x,U)$. We deduce that $\Nc_f(x,U) \in b\Ac = \Ac$ and hence that $x \in \Ac\Tran(f)$. By Lemma~\ref{Lem:point-A-rec/tran} we can now obtain the already announced (and as far as we know completely new) result:

\begin{theorem}\label{The:point-A-tran}
	Let $\Ac \subset \Part(\NN_0)$ be a block Furstenberg family. Given a continuous map $f:X\longrightarrow X$ on a separable complete metric space $(X,d)$, the following are equivalent:
	\begin{enumerate}[{\em(i)}]
		\item $(X,f)$ is weakly-mixing and point-$\Ac$-transitive;
		
		\item $(X^N,f_{(N)})$ is point-$\Ac$-transitive for every $N \in \NN$;
		
		\item $(\Kc(X),\com{f})$ is point-$\Ac$-transitive;
		
		\item $(\Fc_{\infty}(X),\fuz{f})$ is topologically transitive and point-$\Ac$-recurrent;
		
		\item $(\Fc_{0}(X),\fuz{f})$ is point-$\Ac$-transitive;
			
		\item $(\Fc_{S}(X),\fuz{f})$ is point-$\Ac$-transitive;
			
		\item $(\Fc_{E}(X),\fuz{f})$ is point-$\Ac$-transitive.
	\end{enumerate}
\end{theorem}
\begin{proof}
	(i) $\Leftrightarrow$ (ii): The implication (i) $\Rightarrow$ (ii) follows from the separability and completeness of $(X,d)$ together with statement (c) of Lemma~\ref{Lem:point-A-rec/tran} and \cite[Proposition~3.6]{CarMur2025_RACSAM_frequently}, while (ii) $\Rightarrow$ (i) is trivial.
	
	(i) $\Rightarrow$ (iii),(iv): These implications follow from statements (c), (d) and (e) of Lemma~\ref{Lem:point-A-rec/tran} together with part (b) of Corollary~\ref{Cor:point-rec/tran} and \cite[Proposition~3.6]{CarMur2025_RACSAM_frequently}.
	
	(iv) $\Rightarrow$ (v) $\Rightarrow$ (vi) $\Rightarrow$ (vii): This follows from part (b) of Corollary~\ref{Cor:point-rec/tran} together with the topology inclusions $\tau_{E} \subset \tau_{S} \subset \tau_{0} \subset \tau_{\infty}$ and \cite[Proposition~3.6]{CarMur2025_RACSAM_frequently}.
	
	(iii),(vii) $\Rightarrow$ (i): These implications follow from statements (a) and (b) of Lemma~\ref{Lem:point-A-rec/tran} together with Theorem~\ref{The:A-tran} for the family $\Ac_{\infty}$ formed by the infinite subsets of $\NN_0$.
\end{proof}

\begin{remark}\label{Rem:block}
	Let us mention some examples of block families and applications of Theorem~\ref{The:point-A-tran}:
	\begin{enumerate}[(1)]
		\item The family $\Ac_{\neq\varnothing}$ formed by all non-empty subsets of $\NN_0$ and also the already mentioned family~$\Ac_{\infty}$ formed by the infinite subsets of $\NN_0$ are easily checked to be block families. It is worth noticing that {\em point-$\Ac_{\neq\varnothing}$-transitivity} is exactly equal to the usual notion of {\em point-transitivity}. Moreover, if a metric space $(X,d)$ is a singleton or has no isolated points, then {\em point-$\Ac_{\neq\varnothing}$-transitivity} also coincides with the notion of {\em point-$\Ac_{\infty}$-transitivity} for every dynamical system $(X,f)$. Thus, since $(X,d)$ is either a singleton or has no isolated points whenever $(X,f)$ is weakly-mixing, the version of Theorem~\ref{The:point-A-tran} for any of the families $\Ac_{\neq\varnothing}$ or $\Ac_{\infty}$ was already proved in part (b) of Corollary~\ref{Cor:point-rec/tran}.
		
		\item The family $\AP$ formed by the sets that contain arbitrarily long arithmetic progressions (i.e.\ the sets $A \subset \NN_0$ such that given $\ell \in \NN$ there exist $n_0 \in \NN_0$ and $n \in \NN$ with $\{ n_0 + jn \ ; \ 0\leq j\leq \ell \} \subset A$)~is also easily checked to be a block family. Note that along this paper we have already used the notion of {\em point-$\AP$-recurrence} in Corollary~\ref{Cor:point-rec/tran}. Indeed, by \cite[Proposition~3.6]{CarMur2025_RACSAM_frequently} we have that the version of Theorem~\ref{The:point-A-tran} for the family $\AP$ is a consequence of both parts (a) and (b) of Corollary~\ref{Cor:point-rec/tran}.
		
		\item The family of positive upper Banach density sets $\BDsup := \{ A \subset \NN_0 \ ; \ \Bdsup(A)>0 \}$, where
		\[
		\Bdsup(A) := \limsup_{N\to\infty} \left( \max_{m\in\NN_0} \frac{\#(A \cap [m,m+N])}{N+1} \right) \quad \text{ for each } A \subset \NN_0,
		\]
		is also easily checked to be a block family (see \cite[Example~3.5]{CarMur2025_RACSAM_frequently} for further details). The version of Theorem~\ref{The:point-A-tran} for the family $\BDsup$ is, as far as we know, a completely new result in the literature. The {\em $\BDsup$-transitive points} have been called {\em points with reiteratively dense orbit} but also {\em reiteratively hypercyclic points} (see for instance \cite{BonillaGrosseLoPe2022_JFA_frequently,GrivauxLo2023_JMPA_recurrence,GrivauxLoPe2025_BJMA_questions-II,Lopez2024_RIMA_invariant,LopezMe2025_JMAA_two}). We consider worth recalling here that there are {\em topologically mixing} systems that are not {\em point-$\BDsup$-transitive}, i.e.\ there is no general Baire category argument to show the existence of a $\BDsup$-transitive point (see \cite[Theorem~13]{BesMePePu2016_MA_recurrence}).
	\end{enumerate}
\end{remark}
		
Using Remark~\ref{Rem:block}, we end this section by showing an application of Theorem~\ref{The:point-A-tran} in the context of {\em linear dynamical systems}, where we let $f=T$ be a {\em continuous linear operator} acting on a usually {\em infinite-dimensional Fr\'echet space} $X$ (i.e.\ a completely metrizable locally convex vector space). In this setting it is well-known that {\em point-$\BDsup$-transitivity} (which has been called {\em reiterative hypercyclicity}) implies the notion of weak-mixing (see \cite[Proposition~4]{BesMePePu2016_MA_recurrence}). Thus, by Theorem~\ref{The:point-A-tran} we get the following:

\begin{corollary}\label{Cor:point-BD-tran}
	Let $T:X\longrightarrow X$ be a continuous linear operator acting on a Fr\'echet space $X$. Then, the following statements are equivalent:
	\begin{enumerate}[{\em(i)}]
		\item $(X,T)$ is point-$\BDsup$-transitive (i.e.\ reiteratively hypercyclic);
		
		\item $(\Kc(X),\com{T})$ is point-$\BDsup$-transitive (i.e.\ reiteratively hypercyclic);
		
		\item $(\Fc_{0}(X),\fuz{T})$ is point-$\BDsup$-transitive (i.e.\ reiteratively hypercyclic);
		
		\item $(\Fc_{S}(X),\fuz{T})$ is point-$\BDsup$-transitive (i.e.\ reiteratively hypercyclic);
		
		\item $(\Fc_{E}(X),\fuz{T})$ is point-$\BDsup$-transitive (i.e.\ reiteratively hypercyclic).
	\end{enumerate}
\end{corollary}

We do not know if Theorem~\ref{The:point-A-tran} holds for non-block families. Indeed, for the highly studied family of positive lower density sets $\Dinf := \{ A \subset \NN_0 \ ; \ \dinf(A)>0 \}$, where
\[
\dinf(A) := \liminf_{N\to\infty} \frac{\#(A \cap [0,N])}{N+1} \quad \text{ for each } A \subset \NN_0,
\]
it is not even known if {\em point-$\Dinf$-transitivity} for a linear system $(X,T)$ implies or not the respective {\em point-$\Dinf$-transitivity} for $(X^2,T_{(2)})$. This would be a first crucial step to get a similar result to that of Corollary~\ref{Cor:point-BD-tran} for {\em point-$\Dinf$-transitivity} (also called {\em frequent hypercyclicity} in the literature, see \cite{BayartMa2009_book_dynamics,GrossePe2011_book_linear,Lopez2026_MJOM_frequently}).

\section{Devaney chaos and the specification property}\label{Sec_4:chaos_sp}

In this section we focus on Devaney chaos and the specification property, which were considered in the references \cite{MartinezPeRo2021_MAT_chaos} and \cite{BartollMaPeRo2022_AXI_orbit} for fuzzy dynamical systems as described in this paper. In this section, contrary to what we have done in the previous part of this paper, we do not obtain new results. In fact, we only extend the already existent results to the case of the sendograph and endograph metrics.

\subsection{Dense sets of periodic points and Devaney chaos}

Although there is no unified concept of chaos, the notion of Devaney chaos is one of the most usual definitions from those considered in the literature. Given a dynamical system $(X,f)$:
\begin{enumerate}[--]
	\item a point $x \in X$ is called {\em periodic for $f$} if there is some $p \in \NN$ such that $f^p(x)=x$. The minimum positive integer $p \in \NN$ fulfilling that $f^p(x)=x$ is called the {\em period of $x$}, and we will denote by $\Per(f)$ the {\em set of periodic points for $f$}.
	
	\item we say that $(X,f)$ is {\em Devaney chaotic} if it is a topologically transitive system and also the set of periodic points $\Per(f)$ is dense in the space $X$.
\end{enumerate}
It was noticed in \cite[Lemma~1]{MartinezPeRo2021_MAT_chaos} that, given a dynamical system $(X,f)$, then a compact set $K \in \Kc(X)$ is periodic for $\com{f}$ with period $p \in \NN$ if and only if the fuzzy set $\chi_K \in \Fc(X)$ is periodic for $\fuz{f}$ with the same period $p \in \NN$. Using our Lemma~\ref{Lem:key} we can now extend \cite[Proposition~2]{MartinezPeRo2021_MAT_chaos}:

\begin{lemma}\label{Lem:periodic}
	Let $f:X\longrightarrow X$ be a continuous map on a metric space $(X,d)$. Then, the following statements are equivalent:
	\begin{enumerate}[{\em(i)}]
		\item the set of periodic points $\Per(\com{f})$ is dense in $(\Kc(X),d_H)$;
		
		\item the set of periodic points $\Per(\fuz{f})$ is dense in $\Fc_{\infty}(X)$;
		
		\item the set of periodic points $\Per(\fuz{f})$ is dense in $\Fc_{0}(X)$;
		
		\item the set of periodic points $\Per(\fuz{f})$ is dense in $\Fc_{S}(X)$;
		
		\item the set of periodic points $\Per(\fuz{f})$ is dense in $\Fc_{E}(X)$.
	\end{enumerate}
\end{lemma}
\begin{proof}
	(i) $\Leftrightarrow$ (ii) $\Leftrightarrow$ (iii): This was proved in~\cite[Proposition~2]{MartinezPeRo2021_MAT_chaos}.
	
	(iii) $\Rightarrow$ (iv) $\Rightarrow$ (v): This trivially follows from the inclusions $\tau_{E} \subset \tau_{S} \subset \tau_{0}$.
	
	(v) $\Rightarrow$ (i): Given any set $K \in \Kc(X)$ and any $\eps>0$ we must show that $\Bc_H(K,\eps) \cap \Per(\com{f}) \neq \varnothing$. Without loss of generality we will assume that $0<\eps\leq\frac{1}{2}$. Using that $\Per(\fuz{f})$ is dense in $\Fc_{E}(X)$ we can find some fuzzy set $u \in \Fc(X)$ such that
	\[
	u \in \Bc_{E}(\chi_K,\eps) \cap \Per(\fuz{f}).
	\]
	Letting $\delta := d_{E}(\chi_K,u) < \eps \leq \frac{1}{2}$ and picking any $\alpha \in \ ]\delta,1-\delta]$, Lemma~\ref{Lem:key} tells us that $d_H\left( K , u_{\alpha} \right) < \eps$. Moreover, since $u \in \Per(\fuz{f})$ there exists some $p \in \NN$ such that $\fuz{f}^p(u)=u$, which implies that
	\[
	\com{f}^p(u_{\alpha}) = f^p(u_{\alpha}) = [\fuz{f}^{p}(u)]_{\alpha} = u_{\alpha}.
	\]
	We conclude that $u_{\alpha} \in \Bc_H(K,\eps) \cap \Per(\com{f})$, finishing the proof.
\end{proof}

From the definition of Devaney chaos, using Theorem~\ref{The:A-tran} for the family $\Ac_{\infty}$ together with Lemma~\ref{Lem:periodic} and the original results \cite[Corollary~1~and~Theorem~1]{MartinezPeRo2021_MAT_chaos} we get that:

\begin{theorem}\label{The:Devaney}
	Let $f:X\longrightarrow X$ be a continuous map on a metric space $(X,d)$. Then, the following statements are equivalent:
	\begin{enumerate}[{\em(i)}]
		\item $(\Kc(X),\com{f})$ is Devaney chaotic;
		
		\item $(\Fc_{\infty}(X),\fuz{f})$ is Devaney chaotic;
		
		\item $(\Fc_{0}(X),\fuz{f})$ is Devaney chaotic;
		
		\item $(\Fc_{S}(X),\fuz{f})$ is Devaney chaotic;
		
		\item $(\Fc_{E}(X),\fuz{f})$ is Devaney chaotic.
	\end{enumerate}
	Moreover, if $f:X\longrightarrow X$ is a continuous linear operator and $X$ is a Fr\'echet space, then the previous statements are equivalent to:
	\begin{enumerate}[{\em(i)}]
		\item[{\em(vi)}] $(X,f)$ is Devaney chaotic.
	\end{enumerate}
\end{theorem}

As mentioned in~\cite[Section~2]{MartinezPeRo2021_MAT_chaos} we must recall that statement (vi) of Theorem~\ref{The:Devaney} is not equivalent to, nor weaker or stronger than, the other statements of Theorem~\ref{The:Devaney} outside the linear setting. The reader is also referred to \cite[Remark~13~and~Theorem~14]{GarciaKwietLamOPe2009_NA_chaos} for further details.

\subsection{The specification property}

The specification property is a strong dynamical behaviour consisting in the possibility of approximate arbitrary pieces of orbits by segments of a single (usually periodic) orbit. Although there are several versions of the specification property, following \cite{BartollMaPeRo2022_AXI_orbit} we will only be interested in the strongest, namely, the periodic one (see \cite{BartollMaPeRo2022_AXI_orbit,GarciaKwietLamOPe2009_NA_chaos} and the references cited there for further details):
\begin{enumerate}[--]
	\item we say that a dynamical system $(X,f)$ on a metric space $(X,d)$ {\em has the specification property} if: for any $\eps>0$ we can find a positive integer $N_{\eps} \in \NN$ such that, for any integer $s\geq 2$, any finite sequence of points $\{y_1,...,y_s\} \subset X$, and any sequence of integers $0 = i_1 \leq j_1 < i_2 \leq j_2 < \cdots < i_s \leq j_s$ fulfilling that $i_{r+1} - j_r \geq N_{\eps}$ for all $1\leq r<s$, there exists a point $x \in X$ for which
	\begin{equation*}
		d(f^j(x),f^j(y_r))<\eps \text{ for every } 1\leq r\leq s \text{ and } i_r\leq j\leq j_r, \quad \text{ and } \quad f^{N_{\eps}+j_s}(x) = x.
	\end{equation*}
\end{enumerate}

This property may seem quite technical but it is satisfied by many significant dynamical systems in the literature, with shift-like systems being a prominent example. Bauer and Sigmund proved that
\begin{enumerate}[--]
	\item \cite[Proposition~4]{BauerSig1975_MM_topological}: \textit{when $f:X\longrightarrow X$ is a continuous map on a \textbf{compact} metric space $(X,d)$, if the dynamical system $(X,f)$ has the specification property then so does the extended system $(\Kc(X),\com{f})$};
\end{enumerate}
while the converse is not true, since in \cite{GarciaKwietLamOPe2009_NA_chaos} it was proved that
\begin{enumerate}[--]
	\item \cite[Theorem~14]{GarciaKwietLamOPe2009_NA_chaos}: \textit{there exists a dynamical system $(X,f)$ on a \textbf{compact} metric space $(X,d)$ such that $(\Kc(X),\com{f})$ has the specification property while $(X,f)$ does not have the specification property.}
\end{enumerate}
We briefly discuss why \textbf{compactness} can be dropped from~\cite[Proposition~4]{BauerSig1975_MM_topological} in~Corollary~\ref{Cor:sp} below, where following \cite{BartollMaPeRo2022_AXI_orbit} we also give some sufficient conditions on the space $X$ and the map $f$ to get the specification property equivalence between the systems $(X,f)$ and $(\Kc(X),\com{f})$. However, let us start by using our Lemma~\ref{Lem:key} to extend \cite[Theorem~2]{BartollMaPeRo2022_AXI_orbit} for the sendograph and endograph metrics:

\begin{theorem}\label{The:sp}
	Let $f:X\longrightarrow X$ be a continuous map acting on a metric space $(X,d)$. Then, the following statements are equivalent:
	\begin{enumerate}[{\em(i)}]
		\item $(\Kc(X),\com{f})$ has the specification property;
		
		\item $(\Fc_{\infty}(X),\fuz{f})$ has the specification property;
		
		\item $(\Fc_{0}(X),\fuz{f})$ has the specification property;
		
		\item $(\Fc_{S}(X),\fuz{f})$ has the specification property;
		
		\item $(\Fc_{E}(X),\fuz{f})$ has the specification property.
	\end{enumerate}
\end{theorem}
\begin{proof}
	(i) $\Leftrightarrow$ (ii) $\Leftrightarrow$ (iii): This was proved in~\cite[Theorem~2]{MartinezPeRo2021_MAT_chaos}.
	
	(iii) $\Rightarrow$ (iv) $\Rightarrow$ (v): This trivially follows from the inclusions $\tau_{E} \subset \tau_{S} \subset \tau_{0}$ or, more precisely, from the well-known inequalities $d_{E}(u,v) \leq d_{S}(u,v) \leq d_{0}(u,v) \leq d_{\infty}(u,v)$, which hold for all $u,v \in \Fc(X)$.
	
	(v) $\Rightarrow$ (i): By hypothesis $(\Fc_{E}(X),\fuz{f})$ has the specification property. Then, for any $\eps>0$ there exists $\fuz{N}_{\eps} \in \NN$ such that for any integer $s\geq 2$, any finite family of fuzzy sets $\{u^1,...,u^s\} \subset \Fc(X)$, and any sequence of integers $0 = i_1 \leq j_1 < i_2 \leq j_2 < \cdots < i_s \leq j_s$ with $i_{r+1} - j_r \geq \fuz{N}_{\eps}$ for all $1\leq r<s$, there exists a point $v \in \Fc(X)$ for which
	\begin{equation}\label{eq:sp}
		d_E(\fuz{f}^j(v),\fuz{f}^j(u^r)) < \eps \text{ for every } 1\leq r\leq s \text{ and } i_r \leq j\leq j_r, \quad \text{ and } \quad \fuz{f}^{\fuz{N}_{\eps}+j_s}(v)=v.
	\end{equation}
	Let us check that this implies the specification property for $(\Kc(X),\com{f})$. In fact, given $\eps>0$, which can be assumed to fulfill that $0<\eps\leq\tfrac{1}{2}$ without loss of generality, set $\com{N}_{\eps} = \fuz{N}_{\eps} \in \NN$. Now, consider an integer $s\geq 2$, a finite family of non-empty compact sets $\{K_1,...,K_s\} \subset \Kc(X)$ and an arbitrary but fixed sequence of integers $0 = i_1 \leq j_1 < i_2 \leq j_2 < \cdots < i_s \leq j_s$ with $i_{r+1} - j_r \geq \com{N}_{\eps} = \fuz{N}_{\eps}$ for all $1\leq r<s$. Considering now the characteristic function $u^r:=\chi_{K_r}$ for each $1\leq r\leq s$, and applying the specification property to the set $\{u^1,...,u^s\} \subset \Fc(X)$ for the sequence of integers fixed above, we know that there exists $v \in \Fc(X)$ satisfying \eqref{eq:sp}. Let
	\[
	\delta := \max\{ d_E(\fuz{f}^j(v),\fuz{f}^j(u^r)) \ ; \ 1\leq r\leq s \text{ and } i_r \leq j\leq j_r \} < \eps \leq \tfrac{1}{2}
	\]
	and fix any $\alpha \in \ ]\delta,1-\delta]$. Since $\fuz{f}^j(u^r)=\fuz{f}^j(\chi_{K_r})=\chi_{f^j(K_r)}$ for every $1\leq r\leq s$ and $i_r \leq j\leq j_r$, an application of Lemma~\ref{Lem:key} shows that
	\[
	d_H(\com{f}^j(v_{\alpha}),\com{f}^j(K_r)) = d_H( [\fuz{f}^j(v)]_{\alpha} , f^j(K_r) ) < \eps \quad \text{ for every } 1\leq r\leq s \text{ and } i_r\leq j\leq j_r.
	\]
	Moreover, since $\fuz{f}^{\com{N}_{\eps}+j_s}(v)=\fuz{f}^{\fuz{N}_{\eps}+j_s}(v)=v$ we also have that
	\[
	\com{f}^{\com{N}_{\eps}+j_s}(v_{\alpha}) = f^{\com{N}_{\eps}+j_s}(v_{\alpha}) = [\fuz{f}^{\com{N}_{\eps}+j_s}(v)]_{\alpha} = v_{\alpha} \in \Kc(X).
	\]
	We conclude that $(\Kc(X),\com{f})$ has the specification property.
\end{proof}

From Theorem~\ref{The:sp} we obtain the following result improving \cite[Corollary~1~and~Theorem~3]{BartollMaPeRo2022_AXI_orbit}:

\begin{corollary}\label{Cor:sp}
	Let $f:X\longrightarrow X$ be a continuous map acting on a metric space $(X,d)$. Hence:
	\begin{enumerate}[{\em(a)}]
		\item If the dynamical system $(X,f)$ has the specification property, then so do all the extended dynamical systems $(\Kc(X),\com{f})$, $(\Fc_{\infty}(X),\fuz{f})$, $(\Fc_{0}(X),\fuz{f})$, $(\Fc_{S}(X),\fuz{f})$, and $(\Fc_{E}(X),\fuz{f})$.
		
		\item If $f:=T\res_X$, where $T:E\longrightarrow E$ is a continuous linear operator acting on a Fr\'echet space $E$ and $X \subset E$ is a convex $T$-invariant compact subset, then the following statements are equivalent:
		\begin{enumerate}[{\em(i)}]
			\item $(X,f)$ has the specification property;
			
			\item $(\Kc(X),\com{f})$ has the specification property;
			
			\item $(\Fc_{\infty}(X),\fuz{f})$ has the specification property;
			
			\item $(\Fc_{0}(X),\fuz{f})$ has the specification property;
			
			\item $(\Fc_{S}(X),\fuz{f})$ has the specification property;
			
			\item $(\Fc_{E}(X),\fuz{f})$ has the specification property.
		\end{enumerate}
	\end{enumerate}
\end{corollary}
\begin{proof}
	(a): The fact that the dynamical system $(\Kc(X),\com{f})$ has the specification property was proved in \cite[Theorem~2~and~Proposition~4]{BauerSig1975_MM_topological} for the particular case in which $(X,d)$ is a compact metric space. However, their proof can be easily adapted to general metric spaces by properly using the continuity of the map $\com{f}$ rather than relying on its uniform continuity, which is only available in the compact case. The cases of $(\Fc_{\infty}(X),\fuz{f})$, $(\Fc_{0}(X),\fuz{f})$, $(\Fc_{S}(X),\fuz{f})$ and $(\Fc_{E}(X),\fuz{f})$ follow from Theorem~\ref{The:sp}.
	
	(b): This follows from Theorem~\ref{The:sp} together with \cite[Theorem~3]{BartollMaPeRo2022_AXI_orbit}. See also \cite{BernardesPeRo2017_IEOT_set-valued}.
\end{proof}

\section{Conclusions}\label{Sec_5:conclusions}

In this paper, we have examined the interaction between several properties of a discrete dynamical system $(X,f)$ and its extension $(\Fc(X),\fuz{f})$ on the space of normal fuzzy sets $\Fc(X)$ and with respect to the supremum, Skorokhod, sendograph and endograph metrics. In particular, using Lemma~\ref{Lem:key} we have extended the main results from \cite{AlvarezLoPe2025_FSS_recurrence,BartollMaPeRo2022_AXI_orbit,JardonSanSan2020_MAT_transitivity,MartinezPeRo2021_MAT_chaos}. Actually, we have improved the already existing equivalences for the notions of {\em transitivity}, {\em recurrence}, {\em Devaney chaos} and the {\em specification property}, by adding the systems $(\Fc_{S}(X),\fuz{f})$ and $(\Fc_{E}(X),\fuz{f})$, and for separable complete metric spaces we have also obtained some completely new results in terms of {\em point-$\Ac$-transitivity}. Several questions and comments arise naturally from the theory developed in this paper as we comment in the next lines.

A first natural question is why we consider the space $\Fc(X)$ of normal fuzzy sets. This question is relevant because many papers have studied the dynamics of the fuzzification of a system $(X,f)$, but taking alternative spaces of fuzzy sets. For instance, if for a metric space $(X,d)$ we set
\[
\FF(X) := \left\{ u : X \longrightarrow [0,1] \ ; \ u \text{ is upper-semicontinuous and } u_0 \text{ is compact} \right\},
\]
then the spaces of fuzzy sets
\[
\Fc^{=\lambda}(X) := \left\{ u \in \FF(X) \ ; \ \max\{ u(x) \ ; \ x \in X \} = \lambda \right\}
\]
\[
\Fc^{\geq\lambda}(X) := \left\{ u \in \FF(X) \ ; \ u(x) \geq \lambda \text{ for some } x \in X \right\}
\]
were considered in \cite{Kupka2011_IS_on,WuDingLuWang2017_IJBC_topological} for different values $\lambda \in \ ]0,1]$. However, it is not hard to check that the space $\Fc^{=\lambda}(X)$ is homeomorphic to the space of normal fuzzy sets $\Fc(X)$ considered in this paper for every $\lambda \in \ ]0,1]$, and one of the main results in the theory is that the system $(\Fc^{\geq\lambda}(X),\fuz{f})$ cannot be transitive for $\lambda<1$. These considerations led us to adopt the framework used in \cite{AlvarezLoPe2025_FSS_recurrence,BartollMaPeRo2022_AXI_orbit,JardonSan2021_FSS_expansive,JardonSan2021_IJFS_sensitivity,JardonSanSan2020_FSS_some,JardonSanSan2020_MAT_transitivity,MartinezPeRo2021_MAT_chaos}.

A second natural question concerns the dynamical properties studied in \cite{BartollMaPeRo2022_AXI_orbit,JardonSan2021_FSS_expansive,JardonSan2021_IJFS_sensitivity,JardonSanSan2020_FSS_some,MartinezPeRo2021_MAT_chaos}, mentioned at the Introduction, but not addressed in this paper. Actually, from the list of dynamical properties mentioned at the beginning of this paper, we have only focused on those properties for which the endograph metric $d_{E}$ behaves exactly equal to the supremum, Skorokhod and sendograph metrics:~we have already obtained some results for the {\em expansive}, {\em sensitive} and {\em contractive} properties, for {\em Li-Yorke} and {\em distributional chaos}, and for the {\em shadowing property} (among other dynamical notions), but the endograph metric~$d_{E}$ behaves in an extremely radical way for these notions. Our results concerning these dynamical properties will appear in the forthcoming works \cite{Lopez2026_JIA_topological-II} and \cite{AlvarezLo2026_IS_Li-Yorke}.

Finally, regarding possible future research directions, one may study the dynamical properties of the system $(\Fc(X),\fuz{f})$ when $\Fc(X)$ is endowed with other metrics (see \cite{Huang2022_FSS_some}), as well as extending the analysis to the case where $X$ is a (not necessarily metrizable) uniform space (see \cite{JardonSanSan2023_IJFS_fuzzy,JardonSanSan2026_TA_transitivity}).

\section*{Funding}

The author of this paper was supported by MCIN/AEI/10.13039/501100011033/FEDER, UE, Project PID2022-139449NB-I00.

\section*{Acknowledgments}

The author wants to thank Salud Bartoll, Nilson Bernardes Jr., F\'elix Mart\'inez-Jim\'enez, Alfred Peris, and Francisco Rodenas, for their valuable advice and careful readings of the manuscript. The author is also grateful to \'Angel Calder\'on-Villalobos and Manuel Sanchis for insightful discussions on the topic.

{\footnotesize

}

\newpage

{\footnotesize
$\ $\\

\textsc{Antoni L\'opez-Mart\'inez}: Universitat Polit\`ecnica de Val\`encia, Institut Universitari de Matem\`atica Pura i Aplicada, Edifici 8E, 4a planta, 46022 Val\`encia, Spain. e-mail: alopezmartinez@mat.upv.es
}

\end{document}